 \documentclass[final,3p,times]{elsarticle}

\usepackage{amssymb}
 \usepackage{amsthm}
\usepackage{amsmath,amssymb,amsopn,amsfonts,mathrsfs,amsbsy,amscd}

\newcommand{\prs}{\langle\;,\;\rangle}
\newcommand{\too}{\longrightarrow}
\newcommand{\om}{\omega}
\newcommand{\esp}{\quad\mbox{and}\quad}

\newcommand{\G}{{\mathfrak{g}}}

\newcommand{\h}{{\mathfrak{h}}}
\newcommand{\ad}{{\mathrm{ad}}}
\newcommand{\Lu}{{\mathrm{L}}}

\newcommand{\tr}{{\mathrm{tr}}}
\newcommand{\ric}{{\mathrm{ric}}}
\newcommand{\Ri}{{\mathrm{Ric}}}
\newcommand{\Ku}{{\mathrm{K}}}

\newcommand{\B}{{\cal B}}

\newcommand{\na}{\nabla}
\newcommand{\wi}{\widetilde}
\newcommand{\al}{\alpha}
\newcommand{\be}{\beta}
\newcommand{\we}{\wedge}

\newcommand{\Ga}{\Gamma}
\newcommand{\e}{\epsilon}

\newcommand{\la}{\lambda}

\newcommand{\de}{\delta}

\font\bb=msbm10

\def\B{\hbox{\bb B}}
\def\R{\hbox{\bb R}}

\def\N{\hbox{\bb N}}

\def\C{\hbox{\bb C}}

\newtheorem{theo}{Theorem}[section]
\newtheorem{pr}{Proposition}[section]
\newtheorem{Le}{Lemma}[section]

\newtheorem{rem}{Remark}

\begin{document}

\begin{frontmatter}


 

\title{  Four-dimensional homogeneous semi-symmetric Lorentzian manifolds}

 \author[label1,label2,label3]{Abderazak Benroumane, Mohamed Boucetta, Aziz Ikemakhen}
 \address[label1]{Universit\'e Cadi-Ayyad\\ Facult\'e des sciences et techniques\\
 BP 549 Marrakech Maroc.\\e-mail: benreman@yahoo.fr
 }
 \address[label2]{Universit\'e Cadi-Ayyad\\
 Facult\'e des sciences et techniques\\
 BP 549 Marrakech Maroc\\e-mail: m.boucetta@uca.ma
 }
 \address[label3]{Universit\'e Cadi-Ayyad\\
  Facult\'e des sciences et techniques\\
  BP 549 Marrakech Maroc\\e-mail: Ikemakhen@fstg-marrakech.ma
  }



\begin{abstract}We give a complete description of semi-symmetric algebraic curvature tensors on a four-dimensional Lorentzian vector space and we use this description to determine
all four-dimensional homogeneous semi-symmetric Lorentzian manifolds.
\end{abstract}


\end{frontmatter}







\section{Introduction}\label{section1}

A pseudo-Riemannian manifold $(M,g)$ is said to be semi-symmetric if its curvature tensor $\Ku$ satisfies $\Ku.\Ku=0$ which is equivalent to
\begin{equation}\label{eq1}
[\Ku(X,Y),\Ku(Z,T)]=\Ku(\Ku(X,Y)Z,T)+\Ku(Z,\Ku(X,Y)T),
\end{equation}for any vector fields $X,Y,Z,T$. Semi-symmetric pseudo-Riemannian manifolds generalize obviously locally symmetric manifolds ($\na \Ku=0$). They also generalize second-order locally symmetric manifolds ($\na^2\Ku=0$ and $\na\Ku\not=0$).  Semi-symmetric Riemannian manifolds have been first investigated by E. Cartan \cite{cartan} and the first example of a semi-symmetric not locally symmetric Riemannian manifold was given by Takagi \cite{takagi}. More recently, Szabo \cite{zabo, zabo1} gave a complete description of these manifolds. In this study, Szabo  used strong results proper to the Riemannian sitting which suggests that a similar  study of  semi-symmetric Lorentzian manifolds is far more difficult. To our knowledge, there are only few results on three dimensional locally homogeneous semi-symmetric Lorentzian manifolds \cite{calvaruso1, calvaruso2} and on second-order locally symmetric Lorentzian manifolds \cite{blanco, senovilla}. While in the Riemannian case every homogeneous semi-symmetric manifold is actually locally symmetric, in the Lorentzian case they are homogeneous semi-symmetric Lorentzian manifolds which are not locally symmetric.
In this paper, we give a complete description of semi-symmetric algebraic curvature tensors on a four-dimensional Lorentzian vector space and we use this description to determine
 all four-dimensional homogeneous semi-symmetric Lorentzian manifolds. More precisely, we will prove the following result.
 \begin{theo}\label{main} Let $M$ be a four-dimensional homogeneous semi-symmetric Lorentzian manifold. Then its Ricci operator is either diagonolizable or satisfies  $\Ri\not=0$ and $\Ri^2=0$ and one of the following situation occurs:
   \begin{enumerate}\item If $\Ri$ has a non null eigenvalue then $M$ is  locally isometric to a Lie group with a left invariant metric or $M$ is Ricci parallel and in this case one of the following situations occurs:
   \begin{enumerate} 
   \item $M$ is a space of  constant curvature,
   \item $M$ is locally isometric to a direct product of a Riemannian and a Lorentzian surfaces of  constant curvatures,
  \item $M$ is locally isometric to a direct product of $\R$ with a three dimensional Riemannian space form,
   \item $M$ is locally isometric to a direct product of $\R$ with a three dimensional Lorentzian space form.

   \end{enumerate} \item If $\Ri$ has only 0 as eigenvalues then $M$ is Ricci flat or Ricci isotropic, i.e., $\Ri\not=0$ and $\Ri^2=0$.

   \end{enumerate} Moreover, in case 1, $M$ is locally symmetric.

   \end{theo}    
   
   This result shows that if a four-dimensional homogeneous semi-symmetric Lorentzian manifold is not locally symmetric then it must be either Ricci flat or Ricci isotropic.   So
   to get a complete description of four-dimensional homogeneous semi-symmetric Lorentzian manifolds, we need to determine  four-dimensional semi-symmetric Lorentzian Lie groups and four-dimensional homogeneous semi-symmetric Ricci flat or Ricci isotropic Lorentzian manifolds. We will determine  the Lie algebras of semi-symmetric Lorentzian Lie groups up to dimension 4. We will also determine
   four-dimensional homogeneous semi-symmetric Ricci flat or  Ricci isotropic Lorentzian manifolds with non trivial isotropy by using the classification of four-dimensional pseudo-Riemannian homogeneous manifolds   given by Komrakov \cite{komrakov}. 
   Let us give briefly a description of our results and the methods we will use.

Let $(V,\prs)$ be a Lorentzian vector space and $\Ku:V\times V\too\mathrm{so}(V)$ a semi-symmetric algebraic curvature tensor, i.e., $\Ku$ satisfies the algebraic Bianchi identity and \eqref{eq1}. Let $\Ri:V\too V$ be its Ricci operator. The main result here (see Propositions \ref{pr1} and \ref{pr2} ) is that $\Ri$ has only real eigenvalues and, if $\la_1,\ldots,\la_r$ are the non null ones then $V$ splits orthogonally
\begin{equation}\label{eq2} V=V_0\oplus V_{\la_1}\oplus\ldots\oplus V_{\la_r}, \end{equation}where $V_{\la_i}=\ker(\Ri-\la_i\mathrm{Id}_V)$ and $V_0=\ker(\Ri)^2$. Moreover, $\dim V_{\la_i}\geq2$, $\Ku(V_{\la_i},V_{\la_j})=\Ku(V_0,V_{\la_i})=0$ for $i\not=j$,  $\Ku(u,v)(V_{\la_i})\subset V_{\la_i}$ and $\Ku(u,v)(V_{0})\subset V_{0}$. This reduces the study of semi-symmetric algebraic curvature tensors to the ones who are Einstein ($\Ri=\la\mathrm{Id}_V$) or the ones who are Ricci isotropic ( $\Ri\not=0$ and $(\Ri)^2=0$). The determination of Einstein or Ricci isotropic semi-symmetric algebraic curvature tensors is easy when $\dim V=2$ or 3. When $\dim V=4$ the situation is more complicated. In Theorem \ref{main1}, we give the Petrov normal forms of   Einstein semi-symmetric algebraic curvature tensors on a four-dimensional Lorentzian vector space and, in Theorem \ref{th22}, we determine the semi-symmetric isotropic Ricci ones. Einstein semi-symmetric curvature tensors on a four-dimensional vector space have Petrov type I or II depending on the vanishing or not of their scalar curvature. Thus  we get all semi-symmetric curvature tensors on  four-dimensional Lorentzian vector spaces which
 would be useful in the study of four-dimensional  semi-symmetric Lorentzian manifolds and in the proof of Theorem \ref{main}. For instance, as a consequence of Theorem \ref{main1} and based on a result in \cite{derd}, we show that   a four-dimensional Einstein  Lorentzian manifold with non null constant scalar curvature is semi-symmetric if and only if it is locally symmetric (see Theorem \ref{theo2}). The second step of our study is the determination of the Lie algebras of semi-symmetric Lorentzian Lie groups up to dimension 4.

 Now, let $G$ be a Lie group endowed with a left invariant semi-symmetric Lorentzian metric. The restriction of the curvature $\Ku$ to the Lie algebra $\G=T_eG$ of $G$ is a semi-symmetric algebraic curvature tensor and, according to \eqref{eq2}, $\G$ splits orthogonally
\[ \G=\G_0\oplus\G_{\la_1}\oplus\ldots\oplus\G_{\la_r}, \]with $\dim\G_{\la_i}\geq2$.
We show (see Proposition \ref{pr7}) that, for any $i,j\in\{1,\ldots,r\}$ and $i\not=j$,
     \begin{equation*}
             {\G_{\la_j}}.\G_{\la_i}\subset \G_{\la_i}, \;  {\G_{\la_i}}.\G_{\la_i}\subset \G_0+\G_{\la_i}, \;
             {\G_0}.\G_{\la_i}\subset \G_{\la_i}, \;
              {\G_0}.\G_0\subset \G_0, \;  {\G_{\la_i}}.\G_0\subset \G_0+\G_{\la_i},
         \end{equation*}where the dot is the Levi-Civita product given by
         \[ 2\langle u.v,w\rangle=\langle [u,v],w\rangle+\langle [w,v],u\rangle+\langle [w,u],v\rangle. \]When $\dim\G=3$ or $\dim\G=4$ then  $\G$ has one of the following types:
                \begin{enumerate}\item[]$(S30\la)$ $\dim\G=3$ and $\G=\G_0\oplus\G_\la$ with $\dim\G_0=1$, $\G_0.\G_\la\subset\G_\la$, $\G_0.\G_0\subset\G_0$ and $\la\not=0$. 
                \item[]$(S3\la)$ $\dim\G=3$ and $\G=\G_\la$ with $\la\not=0$.
                \item[]$(S30)$ $\dim\G=3$ and $\G=\G_0$.
               \item[]$(S4\mu\la)$  $\dim\G=4$ and $\G=\G_\mu\oplus\G_\la$ with $\dim\G_\mu=\dim\G_\la=2$,  $\la\not=\mu$,  $\la\not=0$, $\mu\not=0$, $\G_\mu.\G_\la\subset\G_\la$, $\G_\la.\G_\mu\subset\G_\mu$, $\G_\la.\G_\la\subset\G_\la$ and $\G_\mu.\G_\mu\subset\G_\mu$.
               \item[]$(S40^1\la)$ $\dim\G=4$ and $\G=\G_0\oplus\G_\la$ with $\dim\G_0=1$, $\G_0.\G_\la\subset\G_\la$, $\G_0.\G_0\subset\G_0$ and $\la\not=0$.
               \item[]$(S40^2\la)$  $\dim\G=4$ and $\G=\G_0\oplus\G_\la$ with $\dim\G_\la=2$, $\G_0.\G_\la\subset\G_\la$ and $\G_0.\G_0\subset\G_0$.
                \item[]$(S4\la)$ $\dim\G=4$ and $\G=\G_\la$ with $\la\not=0$.
                      \item[]$(S40)$ $\dim\G=4$ and $\G=\G_0$.
                     \end{enumerate} Moreover, for any type above we have from Sections \ref{section2} and \ref{section3}  an exact expression of the curvature. Therefore, in order to complete the study, we need to determine from the curvature the Levi-Civita product and hence the Lie algebra structure. This will be done in Section \ref{section5}. The determination of the Levi-Civita product from the curvature leads to a quadratic system of equations with many unknowns and, therefore, it is very difficult to solve. To overcome this problem, we will ovoid when it is possible a direct approach.
                     For $(S30\la)$, $(S4\mu\la)$, $(S40^1\la)$ and $(S40^2\la)$, we will manage to ovoid direct computation and we will use some tricks to show that they are  product or semi-direct product of Lie algebras. All these models are locally symmetric. The model $(S3\la)$ corresponds to the Lie algebras of three dimensional Lorentzian Lie groups with non null constant curvature which are   known mostly in the unimodular case. We give their Lie algebras in the nonunimodular case.  The model $(S30)$ when $\Ri=0$ corresponds to the Lie algebras of three dimensional flat Lorentzian Lie groups which were determined in \cite{bou1, bou}.  For $(S30)$ when $\Ri\not=0$, we will use a direct computation which will give us two classes of Lie algebras  both are not locally symmetric. The model $(S4\la)$ is the Lie algebra of an Einstein semi-symmetric Lorentzian Lie group with $\la\not=0$. According to Theorem \ref{theo2}, it is actually locally symmetric. The Lie algebras of Einstein Lorentzian Lie groups were determined by Calvaruso-Zaeim in \cite{calvaruso}, we will use the list obtained in this paper to determine all the Lie algebras of type $(S4\la)$. For $(S40)$ with $\Ri=0$, by using Calvaruso-Zaeim list we will show that the sectional curvature vanishes. When $\Ri\not=0$, the situation is  the more complicated and we will use considerations on the holonomy Lie algebra to simplify the computations. The Lie algebras obtained there are not locally symmetric not even  second-order locally symmetric. 
                        In section \ref{section6}, we give the list of four-dimensional homogeneous semi-symmetric  Ricci flat or Ricci isotropic Lorentzian manifolds by using Komrakov's classification of four-dimensional homogeneous pseudo-Riemannian manifolds \cite{komrakov}. We prove Theorem \ref{main} in the last section. Finally, we must say that this study would be impossible to achieve without the use of a computation software which permitted us to check all our computations and to save a lot of time.

\section{Semi-symmetric curvature tensors on Lorentzian vector spaces} \label{section2}

In this section, we give some general properties of semi-symmetric curvature tensors on Lorentzian vector spaces. We will show that their Ricci tensor has only real eigenvalues  and any semi-symmetric curvature tensors has an unique decomposition as a sum of semi-symmetric curvature tensors which are either Einstein or Ricci isotropic. Symmetric and skew-symmetric endomorphisms on Lorentzian vector spaces play a central role in our study, so we start by  giving some of their properties we shall need through this paper.

\subsection{Symmetric and skew-symmetric endomorphisms on Lorentzian vector spaces}
A \emph{pseudo-Euclidean  vector space } is  a real vector space  of finite dimension $n$
endowed with  a
nondegenerate inner product  of signature $(q,n-q)=(-\ldots-,+\ldots+)$.  When the
signature is $(0,n)$
(resp. $(1,n-1)$) the space is called \emph{Euclidean} (resp. \emph{Lorentzian}). We denote by $\mathrm{so}(V)$ the Lie algebra of skew-symmetric endomorphisms of $(V,\prs)$ and, for any $u,v\in V$, let $u\wedge v$ denote the skew-symmetric endomorphism of $V$ given by
\[ (u\wedge v)(z)=\langle v,z\rangle u -\langle u,z\rangle v. \]
We have, for any $A\in \mathrm{so}(V)$,
\begin{equation}\label{der}[A,u\wedge v]=(Au)\wedge v+u\wedge(Av).\end{equation}

The following well-known result will play a crucial role in our study (see \cite{oneil}).
\begin{theo}\label{reduction}
    Let  $(V,\prs)$ be a Lorentzian
    vector space of dimension $n\geq3$ and $f:V\too V$ 
    a symmetric  endomorphism.
    Then there exists a basis  $\B$ of $V$ such that the matrices of
$f$ and $\prs$ in $\B$ are given by one  of  the following  types:
  \begin{enumerate} \item type $\{\mathrm{diag}\}$:
  $$\mathrm{M}(f,\B)=\mathrm{diag}(\al_1,...,\al_n), \;\; \mathrm{M}(\prs,\B)=\mathrm{diag}(+1,...,+1,-1),$$
    \item  type $\{n-2,z\bar{z}\}$: $$\mathrm{M}(f,\B)=\mathrm{diag}(\al_1,...,\al_{n-2})\oplus \left(%
\begin{array}{cc}
  a & b \\
  -b & a \\
\end{array}%
\right),b\neq0,\;\;\mathrm{M}(\prs,\B)=\mathrm{diag}(+1,...,+1,-1),$$\item type $\{n,\al2\}$:
$$ \mathrm{M}(f,\B)=\mathrm{diag}(\al_1,...,\al_{n-2})\oplus \left(%
\begin{array}{cc}
  \al & 1 \\
  0 & \al \\
\end{array}%
\right),\;\; \mathrm{M}(\prs,\B)=I_{n-2}\oplus \left(%
\begin{array}{cc}
  0 & 1 \\
  1 & 0 \\
\end{array}%
\right),$$\item  type $\{n,\al3\}$: $$\mathrm{M}(f,\B)=\mathrm{diag}(\al_1,...,\al_{n-3})\oplus \left(%
\begin{array}{ccc}
  \al & 1 & 0 \\
  0 & \al & 1 \\ 0&0&\al\\
\end{array}%
\right),\;\; \mathrm{M}(\prs,\B)=I_{n-3}\oplus \left(%
\begin{array}{ccc}
  0 & 0&1 \\
 0& 1 & 0 \\ 1&0&0\\
\end{array}%
\right).$$
\end{enumerate} 
  \end{theo}
  
Representations of solvable Lie algebras in pseudo-Euclidean vector spaces will appear naturally in our study. There are   some of their properties which will be useful later.
 
Let $\G$ be a real solvable Lie algebra, 
  $(V,\prs)$  a pseudo-Euclidean vector space and   $\rho:\G\too\mathrm{so}(V)$  a representation of $\G$. For any $\la\in\G^*$, put $V_\la=\{x\in V,\;\rho(u)x=\la(u)x\;\mbox{for all}\;u\in\G \}.$ The representation $\rho$  is called \emph{indecomposable} if $V$ does not contain any nondegenerate invariant vector subspace. The following result is  well-known.  
\begin{pr}\label{co1}Let $\rho:\G\too\mathrm{so}(V)$ be  a  representation  on an Euclidean vector space. Then $V$ splits orthogonally
$$V=\bigoplus_{i=1}^qE_i\oplus V_0,$$where, for $i=1,\ldots,q$, $E_i$ is an invariant indecomposable 2-dimensional
vector space and there exists $\la_i\in\G^*\setminus\{0\}$ and $(e_i,f_i)$ and orthonormal basis of $E_i$ such that, for any $u\in\G$ the restriction of $\rho(u)$ to $E_i$ is $\la(u)e_i\wedge f_i$. 
In particular, a solvable subalgebra of $\mathrm{so}(V)$ must be abelian.

\end{pr}

 The proof of the following proposition has been given in \cite{bou}.
  \begin{pr}\label{pr4} Let $\G$ be a solvable Lie algebra and
   $\rho:\G\too \mathrm{so}(V)$  an indecomposable Lorentzian representation. Then one of the following cases occurs:
  \begin{enumerate}\item $\dim V=1$ and $V=V_0$.
  \item $\dim V=2$,  there exists $\la\in\G^*\setminus\{0\}$ and   an orthonormal basis $(g,h)$ of $V$ such that  $\langle g,g\rangle=-1$
  and, for any $u\in\G$, $\rho(u)=\la(u)g\wedge h$.
  \item $\dim V\geq3$, there exists $\la\in\G^*$ such that  $V_\la$ is a totally isotropic one dimensional vector space. Moreover, for any $\mu\not=\la$, $V_\mu=\left\{0\right\}$.
  
  \end{enumerate}
  
  \end{pr}
  
  As a consequence of this proposition, we will prove the following result which will be used in the proof of Theorem \ref{main1}.
  
\begin{pr}\label{pr5} Let $V$ be a four-dimensional Lorentzian vector space and $\G$ a  solvable subalgebra of $\mathrm{so}(V)$ of dimension $\geq2$. Then  we have two cases:
  \begin{enumerate}
  
            \item There exists an orthonormal basis $(e,f,g,h)$ with $\langle h,h\rangle=-1$ such that $\G=\mathrm{span}\{f\wedge g+f\wedge h,g\wedge h\}$.
            
            \item There exists an orthonormal basis $(e,f,g,h)$ with $\langle h,h\rangle=-1$ such that $\G\subset\mathrm{span}\{e\wedge f,e\wedge g+e\wedge h,f\wedge g+f\wedge h,g\wedge h\}$.

                            \end{enumerate}

  \end{pr}
  
  \begin{proof} 
  The Lie algebra $\G$ has a natural representation in $\mathrm{so}(V)$ and the space $V$ splits orthogonally as $V=E\oplus W$, where $W$ is an indecomposable Lorentzian vector space and $E$ is an Euclidean vector space. By virtue of Proposition \ref{pr4}, we distinguish four cases:
  \begin{enumerate}\item $\dim W=1$. In this case, $\G$ act trivially on $W$ and, by virtue of Proposition \ref{co1}, there exists $\la\in\G^*$ and an orthonormal basis $(e,f,g)$ of $E$ such that, for any $u\in\G$, $u=\la(u)f\wedge g$. This is impossible since $\dim\G\geq2$. 
   \item $\dim W=2$. In this case, according to Propositions \ref{co1} and \ref{pr4} there exists $\mu,\la\in\G^*$, an orthonormal basis $(e,f)$ of $E$ and orthonormal basis $(g,h)$ of $V$ such that, for any $u\in\G$, $u=\mu(u)e\wedge f+\la(u)g\wedge h$,  hence $\dim\G=2$ and we get the second case.  
   \item $\dim W=3$. In this case $\G$ act trivially on $E$ and, according to Proposition \ref{pr4}, there exists $\la\in\G^*$ such $V_\la=\R d$ is a totally vector subspace of $W$. Let $V_\la^\perp$ its orthogonal in $W$.  The quotient $V_\la^\perp/V_\la$ is and Euclidean 1-dimensional vector space and hence the quotient action of  $\G$ on it trivial. So there exists $\mu\in\G^*$ such that for any $u\in\G$ and any $x\in V_\la^\perp$, $u(x)=\mu(u)d$. Choose an unitary vector $e\in E$, an unitary vector $f$ in $V_\la^\perp$ and an isotropic vector $\bar{d}$ orthogonal to $f$ and satisfying $\langle d,\bar{d}\rangle=1$. Put $g=\frac{1}{\sqrt{2}}(d+\bar{d})$ and $h=\frac{1}{\sqrt{2}}(d-\bar{d})$. $\B=(e,f,g,h)$ is an orthonormal basis of $V$ and, for any $u\in\G$,  $u=-\frac{1}{\sqrt{2}}\mu(u)(f\wedge g+f\wedge h)-\la(u)g\wedge h$.
      So $\dim\G=2$ and we get the first case. 
    \item $ W=V$. In this case,  according to Proposition \ref{pr4}, there exists $\la\in\G^*$ such $V_\la=\R d$ is a totally vector subspace of $V$.   The quotient $V_\la^\perp/V_\la$ is a 2-dimensional Euclidean vector space and we can use Proposition \ref{co1} for the quotient action of  $\G$ on $V_\la^\perp/V_\la$. So there exists $\mu,\nu_1,\nu_2\in\G^*$ and a couple of unitary vectors $(e,f)$ in $V_\la^\perp$ which are orthogonal such that, for any $u\in\G$, $u(e)=\mu(u)f+\nu_1(u)d$  and $u(f)=-\mu(u)e+\nu_2(u)d$. Choose an isotropic vector $\bar{d}$ orthogonal to $e$ and $f$ and satisfying $\langle d,\bar{d}\rangle=1$. Put $g=\frac{1}{\sqrt{2}}(d+\bar{d})$ and $h=\frac{1}{\sqrt{2}}(d-\bar{d})$. $\B=(e,f,g,h)$ is an orthonormal basis of $V$ and one can check easily that, for any $u\in\G$, $u\in\mathrm{span}\{e\wedge f,e\wedge g+e\wedge h,f\wedge g+f\wedge h,g\wedge h\}$.  \qedhere

  \end{enumerate}

  \end{proof}

\subsection{Semi-symmetric curvature tensors on Lorentzian vector spaces}

A {\it curvature tensor} on   a pseudo-Euclidean vector space $(V,\prs)$ is a bilinear map $\mathrm{K}:V\times V\too\mathrm{so}(V)$ satisfying:
\begin{enumerate}\item for any $u,v\in V$, $\mathrm{K}(u,v)=-\mathrm{K}(v,u)$, 
\item for any $u,v,w\in V$, $\mathrm{K}(u,v)w+\mathrm{K}(v,w)u+\mathrm{K}(w,u)v=0$.
\end{enumerate} These relations imply
\begin{equation}\label{k3}
\langle K(a,b)u,v\rangle=\langle K(u,v)a,b\rangle,\quad a,b,u,v\in V.
\end{equation}
The Ricci curvature associated to $\mathrm{K}$ is the symmetric bilinear form on $V$ given by $\mathrm{r}(u,v)=\tr(\tau(u,v))$, where $\tau(u,v):V\too V$ is given by $\tau(u,v)(a)=\mathrm{K}(u,a)v$. The  Ricci operator is the endomorphism $\mathrm{Ric}:V\too V$ given by  $\langle
\mathrm{Ric}(u),v\rangle =\mathrm{r}(u,v)$.  We call $\Ku$ Einstein (resp. Ricci isotropic) if $\Ri=\la\mathrm{Id}_V$ (resp. $\Ri\not=0$ and $\Ri^2=0$).
Put ${\mathfrak{h} }(\Ku)=\mathrm{span}\{\Ku(u,v)/\;u,\;v\in \G\}$.
 
 A curvature tensor $\mathrm{K}$ is called {\it semi-symmetric} if 
 \begin{equation}\label{semi}
 [\mathrm{K}(u,v),\mathrm{K}(a,b)]=\mathrm{K}(\mathrm{K}(u,v)a,b)+\mathrm{K}(a,\mathrm{K}(u,v)b),\quad u,v,a,b\in V.
 \end{equation} It is easy to see that if $\mathrm{K}$ is semi-symmetric then its Ricci operator satisfies:
\begin{equation}\label{semi1} \mathrm{K}(u,v)\circ\mathrm{Ric}=\mathrm{Ric}\circ \mathrm{K}(u,v),\quad u,v\in V.
\end{equation} 
If $\mathrm{K}$ is semi-symmetric then  ${\mathfrak{h} }(\Ku)$ is a Lie subalgebra of $\mathrm{so}(V)$ called  {\it primitive holonomy algebra} of $\Ku$. It is an immediate consequence of \eqref{der} that the operator $\Ku$ given by $\Ku(u,v)=\la u\wedge v$ is semi-symmetric and  $\dim{\mathfrak{h} }(\Ku)=\dim(\mathrm{so}(V))$. The converse is also true.

\begin{pr}\label{constant} Let $\Ku$ be a semi-symmetric { curvature tensor} on   a pseudo-Euclidean vector space $(V,\prs)$. Then $\dim\mathfrak{h}(\Ku)=\dim(\mathrm{so}(V))$ if and only if there exists a constant $\la\not=0$ such that $\Ku(u,v)=\la u\wedge v$, for any $u,v\in V$.

\end{pr}

\begin{proof}Suppose that $\dim\mathfrak{h}(\Ku)=\dim(\mathrm{so}(V))$. If $\dim V=2$ the result is obvious so we suppose  $\dim V\geq3$.
	Let $\B=(e_i)_{i=1}^n$ an orthonormal basis of $V$ with $\e_i=\langle e_i,e_i\rangle=\pm1$.
	Put
	$\Ku(e_m,e_n)e_q=\sum_s\Ga_{mnq}^se_s$. By virtue of \eqref{semi}, for any $m,n,q\in\{1,\ldots,n\}$, 
	\[ [\Ku(e_m,e_n),\Ku(e_n,e_q)]=\Ku(\Ku(e_m,e_n)e_n,e_q)+\Ku(e_n,\Ku(e_m,e_n)e_q)=- \Ku(\Ku(e_n,e_q)e_m,e_n)-\Ku(e_m,\Ku(e_n,e_q)e_n).\]So
	\[ \Ku(\Ku(e_m,e_n)e_n,e_q)+\Ku(e_n,\Ku(e_m,e_n)e_q)+ \Ku(\Ku(e_n,e_q)e_m,e_n)+\Ku(e_m,\Ku(e_n,e_q)e_n)=0. \]
	This is equivalent to
	\[ \sum_s\left(\Ga_{mnn}^s\Ku(e_s,e_q)+(\Ga_{nqm}^s-\Ga_{mnq}^s)\Ku(e_s,e_n)
	-\Ga_{nqn}^s\Ku(e_s,e_m)   \right)=0. \]
	Suppose $m\not=n$ and $q\notin\{m,n\}$. Since $(\Ku(e_i,e_j))_{1\leq i<j\leq d}$ is basis of $\mathrm{so}(V)$,
	 we get, for any $s\notin\{m,n,q\}$, $\Ga_{mnn}^s=\Ga_{mnq}^s-\Ga_{nqm}^s=\Ga_{nqn}^s=0$ and 
	\[ \sum_{s\in\{m,n,q\}}\left(\Ga_{mnn}^s\Ku(e_s,e_q)-\Ga_{mnq}^s\Ku(e_s,e_n)
		+\Ga_{nqm}^s\Ku(e_s,e_n)-\Ga_{nqn}^s\Ku(e_s,e_m)   \right)=0. \]
Note that $\Ga_{mnq}^q=0$ and hence		
\[
\Ga_{mnq}^m\Ku(e_n,e_m)-(\Ga_{nqn}^q+\Ga_{mnn}^m)\Ku(e_q,e_m)
+\Ga_{nqm}^q \Ku(e_q,e_n)=0. \]
This shows that $\Ga_{mnq}^m=0$ and $\Ga_{mnn}^m=\Ga_{qnn}^q$. So we have shown that,
for any $m\not=n$ and  $q\notin\{m,n\}$, 
$$\Ga_{mnq}^m=0,\;\Ga_{mnn}^m=\Ga_{qnn}^q,\;\Ga_{mnn}^s=0\esp \Ga_{mnq}^s=\Ga_{nqm}^s \quad s\notin\{m,n,q\}.$$
Moreover, the algebraic Bianchi identity implies $\Ga_{mnq}^s+\Ga_{nqm}^s+\Ga_{qmn}^s=3\Ga_{mnq}^s=0$ and hence 
$\Ga_{mnq}^s=0$ for $s\notin\{m,n,q\}$. Hence
$$\Ku(e_m,e_n)e_n=\Ga_{mnn}^me_m\esp \Ku(e_n,e_m)e_q=0.$$ Since $\Ga_{mnn}^m=\Ga_{qnn}^q$, we can put $\la_n=\e_n\Ga_{mnn}^m$. Let show that $\la_n$ doesn't depend on $n$. Since $\langle \Ku(e_m,e_n)e_n,e_m\rangle=\langle \Ku(e_n,e_m)e_m,e_n\rangle$, we get $\e_m\Ga_{mnn}^m=\e_n\Ga_{nmm}^n$. This shows that $\la_n=\la_m=\la$. Finally, we get that $\Ku(u,v)=\la u\wedge v$ which completes the proof.
\end{proof}

Remark that if $\dim V=2$ then any curvature tensor is semi-symmetric so we will suppose that $\dim V\geq3$.

 \begin{pr}\label{pr1}
  Let $\mathrm{K}$ be a  curvature tensor on a Lorentzian space $(V,\prs)$ satisfying \eqref{semi1}. Then its Ricci operator is either of type $\{\mathrm{diag}\}$ or $\{n,02\}$. In particular, all its eigenvalues are real.
 \end{pr}
 
 \begin{proof} Since $\mathrm{Ric}$ is a symmetric endomorphism of $(V,\prs)$ then there exists a basis $\B$ of $V$ such that the matrices of $\mathrm{Ric}$ and $\prs$ in $\B$ have one of the forms listed in the statement of  Theorem \ref{reduction}.
 \begin{enumerate}\item Suppose that the matrices of $\mathrm{Ric}$ and $\prs$ are of type  $\{n-2,z\bar{z}\}$. Put $\B=(e_1,\ldots,e_{n-1},e,\overline{e})$. Then, for $i=1,\ldots,n-2$,
 \[ \mathrm{Ric}(e_i)=\al_i e_i,\;\mathrm{Ric}(e)=ae-b\overline{e}\esp
 \mathrm{Ric}(\overline{e})=be+a\overline{e},\quad b\not=0.  \]
 This shows that the sum of the eigenspaces associated to the real eigenvalues of $\mathrm{Ric}$ is $E=\mathrm{span}\{e_1,\ldots,e_{n-2}\}$. From \eqref{semi1}, we can deduce that for any $u,v\in V$, $\mathrm{K}(u,v)$ leaves invariant $E$ and hence its orthogonal $E^\perp=\mathrm{span}\{e,\overline{e} \}$.
 So
 \begin{eqnarray*}
 b&=&\langle \mathrm{Ric}(\overline{e}),e\rangle
 =\langle \mathrm{K}(\overline{e},e)e,e\rangle-\langle \mathrm{K}(\overline{e},\overline{e})e,\overline{e}\rangle+\sum_{i=1}^{n-2}\langle \mathrm{K}(\overline{e},e_i)e,e_i\rangle=0,
 \end{eqnarray*}which  contradicts the fact that $b\not=0$.
 \item Suppose that the matrices of $\mathrm{Ric}$ and $\prs$ are of   type $\{n,\al2\}$. Put  $\B=(e_1,\ldots,e_{n-2},e,\overline{e})$ and remark that, for $i=1,\ldots,n-2$,
  \[ \mathrm{Ric}(e_i)=\al_i e_i,\;\mathrm{Ric}(e)=\al e,\;
  \mathrm{Ric}(\overline{e})= e+\al\overline{e}\esp \langle e,e\rangle=
  \langle \overline{e},\overline{e}\rangle=0,\langle \overline{e},{e}\rangle=1.   \]
  This shows that $\mathrm{Ric}$ has only real eigenvalues and the sum of the  associated eigenspaces is $E=\mathrm{span}\{e,e_1,\ldots,e_{n-2}\}$. From \eqref{semi1},
   we can deduce that for any $u,v\in V$, $\mathrm{K}(u,v)$ leaves invariant $E$. We have then
  \begin{eqnarray*}
  \al &=&\langle \mathrm{Ric}(\overline{e}),{e}\rangle
  =\langle \mathrm{K}(\overline{e},e){e},\overline{e}\rangle+\langle \mathrm{K}(\overline{e},\overline{e}){e},{e}\rangle+\sum_{i=1}^{n-2}\langle \mathrm{K}(\overline{e},e_i){e},e_i\rangle=\langle \mathrm{K}(\overline{e},e){e},\overline{e}\rangle.
  \end{eqnarray*}On the other hand,
  \begin{eqnarray*}
  \langle \mathrm{K}(\overline{e},e){e},\overline{e}\rangle&=&
  \langle \mathrm{K}(\overline{e},e)(\mathrm{Ric}(\overline{e})-\al\overline{e}),\overline{e}\rangle
 =\langle \mathrm{K}(\overline{e},e)\circ\mathrm{Ric}(\overline{e}),\overline{e}\rangle
 \stackrel{\eqref{semi1}}=\langle \mathrm{Ric}\circ\mathrm{K}(\overline{e},e)\overline{e},\overline{e}\rangle
 =\langle \mathrm{K}(\overline{e},e)\overline{e},\mathrm{Ric}(\overline{e})\rangle\\
 &=&\langle \mathrm{K}(\overline{e},e)\overline{e},e+\al \overline{e} \rangle
 =\langle \mathrm{K}(\overline{e},e)\overline{e},{e}\rangle
 =-\langle \mathrm{K}(\overline{e},e){e},\overline{e}\rangle.
  \end{eqnarray*}So $\al=0$.
 \item Suppose that the matrices of $\mathrm{Ric}$ and $\prs$ are of   type $\{n,\al3\}$. Put  $\B=(e_1,\ldots,e_{n-3},e,f,\overline{e})$ and remark that, for $i=1,\ldots,n-3$,
   \[ \mathrm{Ric}(e_i)=\al_i e_i,\;\mathrm{Ric}(e)=\al e,\;
   \mathrm{Ric}(f)= e+\al f\esp \mathrm{Ric}(\overline{e})= f+\al\overline{e} .  \]
   This shows that $\mathrm{Ric}$ has only real eigenvalues and the sum of the  associated eigenspaces is $E=\mathrm{span}\{e,e_1,\ldots,e_{n-3}\}$. From \eqref{semi1},
    we can deduce that for any $u,v\in V$, $\mathrm{K}(u,v)$ leaves invariant $E$. We have then
   \begin{eqnarray*}
   \al &=&\langle \mathrm{Ric}(\overline{e}),{e}\rangle
   =\langle \mathrm{K}(\overline{e},e){e},\overline{e}\rangle+
   \langle \mathrm{K}(\overline{e},f){e},f\rangle+\langle \mathrm{K}(\overline{e},\overline{e}){e},{e}\rangle+\sum_{i=1}^{n-3}\langle \mathrm{K}(\overline{e},e_i){e},e_i\rangle=\langle \mathrm{K}(\overline{e},e){e},\overline{e}\rangle.
   \end{eqnarray*}Furthermore,
   \begin{eqnarray*}
   \langle \mathrm{K}(\overline{e},e){e},\overline{e}\rangle&=&
   \langle \mathrm{K}(\overline{e},e)(\mathrm{Ric}(f)-\al f),\overline{e}\rangle\\
  &=&\langle \mathrm{K}(\overline{e},e)\circ\mathrm{Ric}(f),\overline{e}\rangle
  -\al\langle \mathrm{K}(\overline{e},e)f,\overline{e}\rangle\\
  &\stackrel{\eqref{semi1}}=&\langle \mathrm{K}(\overline{e},e)f,\mathrm{Ric}(\overline{e})\rangle
    -\al\langle \mathrm{K}(\overline{e},e)f,\overline{e}\rangle\\
   &=&\langle \mathrm{K}(\overline{e},e)f,f+\al\overline{e} \rangle
       -\al\langle \mathrm{K}(\overline{e},e)f,\overline{e}\rangle\\
       &=&0.
   \end{eqnarray*}So $\al=0$. Thus
   \begin{eqnarray*}
   1&=&\langle\mathrm{Ric}(\overline{e}),f\rangle
   =\langle \mathrm{K}(\overline{e},e){f},\overline{e}\rangle+
      \langle \mathrm{K}(\overline{e},f){f},f\rangle+\langle \mathrm{K}(\overline{e},\overline{e}){f},{e}\rangle+\sum_{i=1}^{n-3}\langle \mathrm{K}(\overline{e},e_i){f},e_i\rangle=
      \langle \mathrm{K}(\overline{e},e){f},\overline{e}\rangle.
   \end{eqnarray*}On the other hand,
   \begin{eqnarray*}
   \langle \mathrm{K}(\overline{e},e){f},\overline{e}\rangle=
   \langle \mathrm{K}(\overline{e},e)\mathrm{Ric}(\overline{e}),\overline{e}\rangle
   \stackrel{\eqref{semi1}}=\langle \mathrm{K}(\overline{e},e)\overline{e},\mathrm{Ric}(\overline{e})\rangle
   =\langle \mathrm{K}(\overline{e},e)\overline{e},f\rangle
   =-\langle \mathrm{K}(\overline{e},e){f},\overline{e}\rangle.
   \end{eqnarray*}This shows that $\langle \mathrm{K}(\overline{e},e){f},\overline{e}\rangle=0$ which contradicts what above and completes the proof.\qedhere

 \end{enumerate}

 \end{proof}

\begin{pr}\label{pr2}Let $\Ku$ be a semi-symmetric curvature tensor on a Lorentzian vector space $(V,\prs)$. Then all eigenvalues of $\mathrm{Ric}$ are real. Denote by $\al_1,\ldots,\al_r$ the non null eigenvalues and $V_1,\ldots,V_r$ the corresponding eigenspaces. Then:
\begin{enumerate}\item $V$ splits orthogonally as $V=V_0\oplus V_1\oplus\ldots\oplus V_r$, where $V_0=\ker(\mathrm{Ric})^2$,
\item for any $u,v\in V$ and $i=0,\ldots,r$, $\Ku(u,v)(V_i)\subset V_i$,
\item for any $i,j=0,\ldots,r$ with $i\not=j$, $\Ku(V_i,V_j)=0$,
\item for any $i=1,\ldots,r$, $\dim V_i\geq2$.

\end{enumerate}

\end{pr}

\begin{proof}\begin{enumerate}\item This is a consequence of Proposition \ref{pr1}.
\item This statement follows from  \eqref{semi1}.
\item Let $u\in V_i$,  $v\in V_j$ and $a,b\in V$. Since $\Ku(a,b)(V_i)\subset V_i$ and $\langle V_i,V_j\rangle=0$, we get
\[ 0=\langle \Ku(a,b)u,v\rangle\stackrel{\eqref{k3}}=\langle \Ku(u,v)a,b\rangle \]and hence $\Ku(u,v)=0$.
\item Suppose that $\dim V_i=1$ for $i=1,\ldots,r$ and choose a generator $e$ of $V_i$ such that $\langle e,e\rangle=\e$ with $\e^2=1$ and complete to get an orthonormal basis $(e,e_1,\ldots,e_{n-1})$ with $\langle e_i,e_i\rangle=\e_i$, $\e_i^2=1$. For any $a,b\in V$, $\Ku(a,b)$ is skew-symmetric and leaves $V_i$ invariant so $\Ku(a,b)e=0$. Now
\[ \e\al_i=\langle \mathrm{Ric}(e),e\rangle=\e\langle \Ku(e,e)e,e\rangle+\sum_{i=1}^{n-1}\e_i\langle \Ku(e,e_i)e,e_i\rangle=0, \]
which is a contradiction and achieves the proof. \qedhere
\end{enumerate}

\end{proof}

This proposition reduces  the determination of semi-symmetric curvature tensors on Lorentzian vector spaces to the determination of three classes of semi-symmetric curvature tensors: Einstein semi-symmetric curvature tensors on an Euclidean vector space, Einstein semi-symmetric curvature tensors on a Lorentzian vector space and Ricci isotropic semi-symmetric curvature tensors on a Lorentzian vector space.

\section{Einstein and Ricci isotropic semi-symmetric curvature tensors on vector spaces of dimension less or equal to 4} \label{section3}

Remark first that if $\dim V=2$ any non null curvature tensor is Einstein semi-symmetric.
If $(V,\prs)$ is a pseudo-Euclidean vector space with $\dim V=3$ then any  curvature tensor is entirely determined by its Ricci operator  and we have
\begin{equation}\label{eq3}
\Ku(X,Y)=\frac12\tr(\Ri)X\wedge Y-X\wedge \Ri(Y)-\Ri(X)\wedge Y.
\end{equation}If $\Ku$ is a curvature tensor on a pseudo-Euclidean vector space of dimension $n\geq4$ then
\begin{equation}\label{weyl}
\Ku(X,Y)=\frac{\tr(\Ri)}{(n-1)(n-2)}X\wedge Y-\frac1{n-2}(X\wedge \Ri(Y)+\Ri(X)\wedge  Y)+W(X,Y)
\end{equation}where $W$ is the Weyl tensor. The following proposition is an immediate consequence of \eqref{eq3}.
\begin{pr}\label{pr3} Let $(V,\prs)$ be a pseudo-Euclidean vector space of dimension $3$ and $\Ku$  a curvature tensor on $(V,\prs)$. 
\begin{enumerate}\item If  $\Ri=\la\mathrm{Id}_V$  then,
for any $u,v\in V$, $\Ku(u,v)=-\frac{\la}2 u\wedge v$ and hence $\Ku$ is semi-symmetric. 
 \item If $V$ is Lorentzian and $\Ku$ is Ricci isotropic semi-symmetric    then there exists a basis $(e,f,g)$ of $V$ such that $\langle e,e\rangle=\langle f,g\rangle=1$, $\Ri(e)=\Ri(f)=0$, $\Ri(g)=f$ and
  $\Ku(e,f)=\Ku(f,g)=0$ and $\Ku(e,g)=-e\we f$.

\end{enumerate}

\end{pr}

The study of Einstein or Ricci isotropic semi-symmetric curvature  tensors    in dimension four is more complicated and needs some preparation. 
Actually,  Einstein curvature tensors on a four-dimensional Lorentzian vector space have been determined by Petrov since 1954 see \cite{petrov} and \cite{besse} pp. 100. This study is based on the fact that if $V$ is a 4-dimensional Lorentzian vector space then $\wedge^2V\simeq\mathrm{so}(V)$ carries a natural complex structure $J=*$, where $*$ is   the Hodge star operator and, a curvature tensor is Einstein if its total curvature operator commutes with $J$. Using this fact, we can find an orthonormal basis of $V$ in which the expression of the curvature tensor has  normal forms called Petrov normal forms. We will recall briefly the construction of Petrov normal forms (for more details see \cite{besse} pp.100) and we will determine the Petrov normal forms of Einstein semi-symmetric curvature tensors.

Let $(V,\prs)$ be a four-dimensional Lorentzian vector space and $\Ku$ a curvature tensor. The inner product  $\prs$ induces naturally an inner product  on any $\wedge^pV$ denoted in the same way. The total curvature operator $\wi\Ku:\wedge^2V\too\wedge^2V$ is given by
   $\langle \wi\Ku(u\wedge v),w\wedge t\rangle=\langle \Ku(u,v)w,t\rangle$. Fix an orientation $\om\in\wedge^4V$ such that $\langle\om,\om\rangle=-1$ and denote by $J:\wedge^2 V\too\wedge^2V$ the Hodge star operator given by
    $\al\wedge\be=\langle J\al,\be\rangle\om$. It is a well-known fact that $J^2=-\mathrm{Id}_{\wedge^2V}$ and hence $J$ induces a complex vector space structure on $\wedge^2V$. If $(e,f,g,h)$ is an orthonormal basis with $\langle h,h\rangle=-1$ and $e\wedge f\wedge g\wedge h=\om$, one can see easily that
    \begin{equation}\label{J}
    J(e\wedge f)=-g\wedge h,\; J(e\wedge g)=f\wedge h\esp J(e\wedge h)=f\wedge g.
    \end{equation}By using these formulas, one can check easily that, for any $\al\in\wedge^2V$, $[J\al,\al]=0$ and hence
    \begin{equation}\label{comm}
    \forall\al,\be\in \wedge^2V,\quad[J\al,J\be]=-[\al,\be].
    \end{equation}Remark that this relation implies that the Nijenhuis torsion of $J$ is given by $N_J(\al,\be)=-2[\al,\be]$.

    In the basis $(e\wedge f,e\wedge g,e\wedge h,f\wedge g,f\wedge h,g\wedge h)$, we have
    \[ [\wi\Ku]=\left(\begin{array}{cccccc}k_{11}&k_{12}&k_{13}&k_{14}&k_{15}&k_{16}\\ 
         k_{12}&k_{22}&k_{23}&k_{24}&k_{25}&k_{26}\\ 
         -k_{13}&-k_{23}&k_{33}&k_{34}&k_{35}&k_{36}\\
         k_{14}&k_{24}&-k_{34}&k_{44}&k_{45}&k_{46}\\
         -k_{15}&-k_{25}&k_{35}&-k_{45}&k_{55}&k_{56}\\
         -k_{16}&-k_{26}&k_{36}&-k_{46}&k_{56}&k_{66}\end{array}\right)\esp
         [J]=\left(\begin{array}{cccccc}0&0&0&0&0&1\\
         0&0&0&0&-1&0\\
         0&0&0&-1&0&0\\
         0&0&1&0&0&0\\
         0&1&0&0&0&0\\
         -1&0&0&0&0&0\end{array}   \right).   
         \]It is a well-known fact that $\Ku$ is Einstein if and only if $\wi\Ku\circ J=J\circ\wi\Ku$ and one can check that this is equivalent to 
         $$\left\{\begin{array}{l}k_{11}=k_{66},k_{22}=k_{55},k_{33}=k_{44}\\
              k_{35}=-k_{24},\;k_{36}=k_{14},\;k_{26}=-k_{15},\\
              k_{56}=-k_{12},\;k_{46}=k_{13},\;k_{45}=-k_{23}.\end{array}   \right.$$
              From the algebraic Bianchi identity we get $k_{16}=k_{25}+k_{34}$ and the relation $\Ri=\la \mathrm{Id}_V$ gives $k_{11}+k_{22}+k_{33}=-\la$. If we put
              $$p=k_{22}-\imath k_{25},\; q=k_{33}-\imath k_{34},\;u=k_{12}-\imath k_{15},\;v=k_{13}-\imath k_{14}\esp w=k_{23}-\imath k_{24},$$we get that the matrix of $\wi\Ku$ (as a complex endomorphism) in the complex basis $\{e\wedge f,e\wedge g,e\wedge h  \}$ is
                 \begin{equation}\label{complex}[\wi\Ku]_{\C}=\left(\begin{array}{ccc}-\la-p-q&u&v\\ 
                 u&p&w\\ 
                 -v&-w&q\end{array}\right).\end{equation}
    There are three   Petrov normal forms of $\wi\Ku$ depending on the type of Jordan normal forms of $\wi\Ku$ as a complex endomorphism. Indeed,   the $\C$-linear endomorphism $\wi\Ku$ has three possible   Jordan normal  forms $$\left(\begin{array}{ccc}z_1&0&0\\ 
                     0&z_2&0\\ 
                     0&0&z_3\end{array} \right),\;\left(\begin{array}{ccc}z_1&0&0\\ 
                                          0&z_2&1\\ 
                                          0&0&z_2\end{array} \right)\quad\mbox{or} \quad\left(\begin{array}{ccc}z&1&0\\ 
                                                               0&z&1\\ 
                                                               0&0&z\end{array} \right)$$ 
    and, for each of them there exists and orthonormal basis $(e,f,g,h)$ of $V$ with $\langle e,e\rangle=-1$ such that the matrix of the $\R$-linear endomorphism $\wi\Ku$ in the basis  $\{e\wedge f,
    e\wedge g,e\wedge h,g\wedge h,h\wedge f,f\wedge g  \}$ has the form $\left( \begin{array}{cc}A&-B\\B&A\end{array}\right)$ where $z_i=\la_i+\imath\mu_i$ and
    \begin{eqnarray}
    \mathrm{Type\; I}& A=\left(\begin{array}{ccc}\la_1&0&0\\ 
                         0&\la_2&0\\ 
                         0&0&\la_3\end{array} \right),\quad B=\left(\begin{array}{ccc}\mu_1&0&0\\ 
                                              0&\mu_2&0\nonumber\\ 
                                              0&0&\mu_3\end{array} \right),\quad\mu_1+\mu_2+\mu_3=0,\\
  \label{petrov}  \mathrm{Type\; II}& A=\left(\begin{array}{ccc}\la_1&0&0\\ 
                             0&\la_2+1&0\\ 
                             0&0&\la_2-1\end{array} \right),\quad B=\left(\begin{array}{ccc}\mu_1&0&0\\ 
                                                                           1&\mu_2&1\\ 
                                                                           0&1&\mu_3\end{array} \right),\quad\mu_1+2\mu_2=0,\\
     \mathrm{Type\; III}& A=\left(\begin{array}{ccc}\la&1&0\\ 
                              1&\la&0\\ 
                              0&0&\la\end{array} \right),\quad B=\left(\begin{array}{ccc}0&0&0\\ 
                                                   0&0&1\\ 
                                                   0&1&0\end{array} \right). \nonumber                                            
    \end{eqnarray}

We can give now a precise description of Einstein semi-symmetric curvature tensors on a four-dimensional Lorentzian vector space  by giving their Petrov normal forms.

\begin{theo}\label{main1} Let $\Ku$ be a non null semi-symmetric curvature tensor on a 4-dimensional Lorentzian vector space $(V,\prs)$ such that $\Ri=\la\mathrm{Id}_V$. Then the primitive holonomy algebra $\mathfrak{h}(\Ku)$ has dimension 2 or 6 and if $\dim\mathfrak{h}(\Ku)=2$  then it is an abelian Lie algebra. Moreover, we have:
\begin{enumerate}\item If $\dim\mathfrak{h}(\Ku)=6$ then for any $u,v\in V$, $\Ku(u,v)=-\frac{\la}3u\wedge v$.

 \item If $\dim\mathfrak{h}(\Ku)=2$ and $\la\not=0$ then $\wi\Ku$ is diagonalizable  as a $\C$-linear endomorphism and its eigenvalues are $-\la$ with complex multiplicity 1 and 0 of complex multiplicity 2 and hence $\wi\Ku$ is of type I.
 \item If $\dim\mathfrak{h}(\Ku)=2$ and $\la=0$  then $\wi\Ku^2=0$ and hence $\wi\Ku$ is of type II with 0 as the only eigenvalue.

\end{enumerate}

\end{theo}

\begin{proof} If $\dim\mathfrak{h}(\Ku)=6$ then the  result follows from Proposition \ref{constant}. Suppose that $\dim\mathfrak{h}(\Ku)\leq5$. Since $\wi\Ku$ commutes with $J$, $\mathfrak{h}(\Ku)$ is invariant by $J$ and hence  $\dim\mathfrak{h}(\Ku)$ is even. Moreover, if $\dim\mathfrak{h}(\Ku)=2$ then the relation \eqref{comm} implies that  $\mathfrak{h}(\Ku)$ is an abelian Lie algebra. When $\dim\mathfrak{h}(\Ku)=4$  it admits a basis of the form $\{A,JA,B,JB \}$, $A,B\in\mathrm{so}(V)$. By using \eqref{comm}, one can see easily that $[\mathfrak{h}(\Ku),\mathfrak{h}(\Ku)]=
\mathrm{span}\{[A,B],[A,JB]  \}$ and 
hence $\mathfrak{h}(\Ku)$  is solvable.
According to Proposition \ref{pr5}, there exists an orthonormal basis $(e,f,g,h)$ such that $\langle h,h\rangle=-1$ and, for any $u,v\in V$, $\Ku(u,v)\in\mathrm{span}\{e\wedge f,e\wedge g+e\wedge h,f\wedge g+f\wedge h,g\wedge h\}$. From \eqref{complex}, we get that the matrix of $\wi\Ku$ in the complex basis $\{e\wedge f,e\wedge g,e\wedge h  \}$ must have the form
\[ [\wi\Ku]_{\C}=\left(\begin{array}{ccc}-\la&u&-u\\u&-v&v\\u&-v&v\end{array}   \right). \]
Suppose $\dim\mathfrak{h}(\Ku)=4$, i.e., the complex rank of $\wi\Ku$ is 2. Since $\Ku(e,h)=-\Ku(e,g)$, $\{\Ku(e,f),\Ku(e,g),\Ku(f,g),\Ku(g,h)  \}$ is a real basis of $\mathfrak{h}(\Ku)$ and we get from \eqref{semi}
    \[ 0=[\Ku(e,g),\Ku(e,h)]=\Ku(\Ku(e,g)e,h)+\Ku(e,\Ku(e,g)h). \] 
    Or
    \[ \Ku(e,g)=(u_1+\imath u_2)e\wedge f-(v_1+\imath v_2)(e\wedge g+e\wedge h)
    =u_1e\wedge f-u_2g\wedge h-v_1(e\wedge g+e\wedge h)-v_2(f\wedge g+f\wedge h). \]So
    \[ -u_1\Ku(e,h)+v_1\Ku(g,h)+u_2\Ku(e,g)+v_2\Ku(e,f)=0. \]Hence
    $u=v=0$ which contradicts the fact that the rank of $[\wi\Ku]_{\C}$ is 2. So $\dim\mathfrak{h}(\Ku)=4$ is impossible.
    
   Suppose now that  $\dim\mathfrak{h}(\Ku)=2$. This is equivalent to $\mathrm{Rank}([\wi\Ku]_{\C})=1$ which is equivalent to $u^2=\la v$. If $\la=0$ then  $\wi\Ku$ is of type II with $\wi\Ku^2=0$. If $\la\not=0$ then
   \[ [\wi\Ku]_{\C}=\left(\begin{array}{ccc}-\la&u&-u\\u&-\la^{-1}u^2&\la^{-1}u^2\\u&-\la^{-1}u^2&\la^{-1}u^2\end{array}   \right). \]It is easy to check that this matrix is diagonalizable with $-\la$ as an eigenvalue of multiplicity 1 and 0 of multiplicity 2. This completes the proof.
   \qedhere

\end{proof}

\begin{rem}\label{remark1} Let $\Ku$ be a semi-symmetric curvature tensor on a four-dimensional Lorentzian vector space with $\Ri=\la\mathrm{Id}_V$. According to Theorem \ref{main1} and the Petrov normal forms given in \eqref{petrov}, there exists an orthonormal basis $(e,f,g,h)$ with $\langle e,e\rangle=-1$ such that:
\begin{enumerate}\item $\la\not=0$, $\Ku(e\wedge f)=-\la e\wedge f$, $\Ku(g\wedge h)=-\la g\wedge h$ and $\Ku(e,g)=\Ku(e,h)=\Ku(f,g)=\Ku(f,h)=0$.
\item $\la=0$, $\Ku(e,g)=-\Ku(f,g)=e\wedge g+f\wedge g$, $\Ku(f,h)=-\Ku(e,h)=e\wedge h+f\wedge h$ and 
$\Ku(e,f)=\Ku(g,h)=0$.
\end{enumerate}

\end{rem}

          Another important consequence of Theorem \ref{main1} is the following result which the proof is based on  Theorem 5.1 in \cite{derd}.
          
          \begin{theo}\label{theo2} Let $M$ be an Einstein  four-dimensional Lorentzian manifold with non null constant scalar curvature. Then $M$ is semi-symmetric if and only if it is locally symmetric.
          
          \end{theo}

       \begin{proof} Suppose that $M$ is semi-symmetric. For any $p\in M$, $\Ku_p$ is a semi-symmetric curvature tensor on $T_pM$. According to Theorem \ref{main1}, its total curvature operator is diagonalizable as $\C$-linear endomorphism of $\wedge^2 T_pM$ with eigenvalues 0 and $-\frac\la4$ with $\la$ is the scalar curvature.  So the eigenvalues are constant and, according to Theorem 5.1 in \cite{derd}, $M$ is locally symmetric.
       \end{proof}

    We end this section by giving  a normal form of Ricci isotropic semi-symmetric  curvature tensors  on 4-dimensional Lorentzian vector spaces. 
    
    \begin{theo}\label{th22} Let $\Ku$ be a Ricci isotropic semi-symmetric curvature tensor on a 4-dimensional Lorentzian vector space $(V,\prs)$. Then  there exists a basis   $(e,f,g,h)$ such that the non vanishing products are $\langle
       e,e\rangle=\langle f,f\rangle=\langle g,h\rangle=1$
     and   \begin{equation*}
     \Ku(e,h)=A e\wedge g,\;
     \Ku(f,h)=B f\wedge g,\;
     \Ku  (e,f)=\Ku(e,g)=\Ku(f,g)=\Ku(g,h)=0,
    \end{equation*}with  $A+B=-1$.
    \end{theo}
    
    \begin{proof}
    According to Proposition \ref{pr1}, there exists a basis $(e,f,g,h)$ with $\langle e,e\rangle=\langle f,f\rangle=\langle g,h\rangle=1$
     such that $\Ri(e)=\Ri(f)=\Ri(g)=0$ and $\Ri(h)=g$. 
    The matrix of $\wi\Ku$ in the basis $(e\wedge f,e\wedge g,e\wedge h,f\wedge g,f\wedge h,g\wedge h)$ is given by
    $$\mathrm{M}=\left(\begin{array}{cccccc}k_{11}&k_{12}&k_{13}&k_{14}&k_{15}&k_{16}\\ 
    k_{13}&k_{22}&k_{23}&k_{24}&k_{25}&k_{26}\\ 
    k_{12}&k_{32}&k_{22}&k_{34}&k_{35}&k_{36}\\
    k_{15}&k_{35}&k_{25}&k_{44}&k_{45}&k_{46}\\
    k_{14}&k_{34}&k_{24}&k_{54}&k_{44}&k_{56}\\
    -k_{16}&-k_{36}&-k_{26}&-k_{56}&-k_{46}&k_{66}\end{array}\right).$$  
   The relations $\Ri(e)=\Ri(f)=\Ri(g)=0$ and $\Ri(h)=g$ are equivalent to
    $$\left\{\begin{array}{l}k_{11}+k_{22}+k_{33}=0,\\
    k_{11}+k_{44}+k_{55}=0,\\
    k_{32}=-k_{54},~~~
    k_{42}=-k_{53},\\
    k_{51}=-k_{62},\;k_{41}=k_{63},\;k_{31}=k_{64},\\
    k_{21}=-k_{65},\;k_{22}=-k_{44},\\
    k_{23}+k_{45}=1.\end{array}   \right..$$   
    So $M$ has the form
    $$\mathrm{M}=\left(\begin{array}{cccccc}0&u&v&w&x&y\\ 
    v&0&b&p&q&-x\\ 
    u&z&0&r&-p&w\\
    x&-p&q&0&1-b&v\\
    w&r&p&-z&0&-u\\
    -y&-w&x&u&-v&a\end{array}\right).$$  
    Thus, in the basis $(e,f,g,h)$, we have
    \[ [\Ku(e,f)]=\left(\begin{array}{cccc}0&0&-u&-v\\
    0&0&-w&-x\\
    v&x&y&0\\u&w&0&-y\end{array}     \right)\esp
    [\Ku(e,g)]=\left(\begin{array}{cccc}0&-u&-z&0\\
    u&0&-r&p\\
    0&-p&w&0\\z&r&0&-w\end{array}     \right). \]
    \[ [\Ku(e,h)]=\left(\begin{array}{cccc}0&-v&0&-b\\
    v&0&-p&-q\\
    b&q&-x&0\\0&p&0&x\end{array}     \right)\esp
    [\Ku(f,g)]=\left(\begin{array}{cccc}0&-w&-r&-p\\
    w&0&z&0\\
    p&0&-u&0\\r&-z&0&u\end{array}     \right). \]
    \[ [\Ku(f,h)]=\left(\begin{array}{cccc}0&-x&p&-q\\
    x&0&0&b-1\\
    q&1-b&v&0\\-p&0&0&-v\end{array}     \right)
    \esp [\Ku(g,h)]=\left(\begin{array}{cccc}0&-y&-w&x\\
    y&0&u&-v\\
    -x&v&-a&0\\w&-u&0&a\end{array}     \right). \]
    By using the relation $\Ku(u,v)\circ \Ri=\Ri\circ \Ku(u,v)$ for the matrices above,  we get
      $y=w=x=u=v=a=z=r=p=0.$
     Therefore, $\Ku(e,h)=-b e\wedge g-qf\wedge g$, $\Ku(f,h)=-q e\wedge g-(1-b)f\wedge g$
     and  $\Ku(e,f)=\Ku(e,g)=\Ku(f,g)=\Ku(g,h)=0$.
      We claim that there exists $\theta$ such that
    \[ \Ku(\cos\theta e+\sin\theta f,h)(-\sin\theta e+\cos\theta f)=
    \Ku(-\sin\theta e+\cos\theta f,h)(\cos\theta e+\sin\theta f)=0. \]
    Indeed, this equivalent to $q\cos(2\theta)+\frac12(1-2\la)\sin(2\theta)=0$ and this equation has obviously a solution.  In the basis $(e',f',g,h)$ with
    $e'=\cos\theta e+\sin\theta f$ and $f'=-\sin\theta e+\cos\theta f$, we have
    $\Ku(e',h)=A e'\wedge g$ and $\Ku(f',h)=B f'\wedge g$. To have $\Ri(h)=g$ we need  $A+B=-1$.
    \end{proof}
    
    \section{Semi-symmetric Lorentzian Lie algebras: general properties} \label{section4}
    
     
     A Lie group $G$ together with a left-invariant pseudo-Riemannian metric $g$ is called a 
     \emph{pseudo-Riemannian Lie group}. The  metric $g$ 
     defines a  pseudo-Euclidean product $\prs$ on the Lie algebra $\G=T_eG$ of $G$, and conversely, any  pseudo-Euclidean product on $\G$
     gives rise
     to an unique  left-invariant pseudo-Riemannian metric on $G$.\\ We will refer to a Lie
     algebra endowed with a  pseudo-Euclidean product as a \emph{pseudo-Euclidean Lie algebra}.  The
     Levi-Civita connection of $(G,g)$ defines a product $\mathrm{L}:\G\times\G\too\G$ called the Levi-Civita product and given by  Koszul's
     formula
     \begin{eqnarray}\label{levicivita}2\langle
     \mathrm{L}_uv,w\rangle&=&\langle[u,v],w\rangle+\langle[w,u],v\rangle+
     \langle[w,v],u\rangle.\end{eqnarray}
     For any $u,v\in\G$, $\mathrm{L}_{u}:\G\too\G$ is skew-symmetric and $[u,v]=\mathrm{L}_{u}v-\mathrm{L}_{v}u$. We will also write $u.v=\mathrm{L}_{v}u$.
     The curvature on $\G$ is given by
     $
      \label{curvature}\Ku(u,v)=\mathrm{L}_{[u,v]}-[\mathrm{L}_{u},\mathrm{L}_{v}].
     $
     It is well-known that $\Ku$ is a curvature tensor on $(\G,\prs)$ and, moreover, it satisfies the differential Bianchi identity
     \begin{equation}\label{2bianchi}
     \mathrm{L}_u(\Ku)(v,w)+\mathrm{L}_v(\Ku)(w,u)+\mathrm{L}_w(\Ku)(u,v)=0,\quad u,v,w\in\G
     \end{equation}where $\mathrm{L}_u(\Ku)(v,w)=[\mathrm{L}_u,\Ku(v,w)]-\Ku(\mathrm{L}_uv,w)-\Ku(v,\mathrm{L}_uw).$
      Denote by $\mathfrak{h}(\G)$ the holonomy Lie algebra of $(G,g)$. It is the smallest Lie algebra containing  $\mathfrak{h}(\Ku)=\mathrm{span}\{\Ku(u,v):u,v\in\G \}$ and satisfying $[\mathrm{L}_u,\mathfrak{h}(\G)]\subset\mathfrak{h}(\G)$, for any $u\in\G$.\\
      If we denote by $\mathrm{R}_{u}:\G\too\G$ the right multiplication given by $\mathrm{R}_{u}v=\mathrm{L}_{v}u$, it is easy to check the following useful relation
           \begin{equation}\label{cu2}\Ku(u,.)v=-\mathrm{R}_v\circ\mathrm{R}_u+\mathrm{R}_{u.v}+[\mathrm{R}_v,\mathrm{L}_u].
           \end{equation} We can also see easily that
           \begin{equation}
            \label{eq7}[\G,\G]^\perp=\{u\in\G,\mathrm{R}_u=\mathrm{R}_u^*\}\esp
            (\G.\G)^\perp=\{u\in\G,\mathrm{R}_u=0\}. 
            \end{equation}
           $(G,g)$ is semi-symmetric iff $\Ku$ is a semi-symmetric curvature tensor of $(\G,\prs)$.  Without reference to any Lie group, we call a pseudo-Euclidean Lie algebra $(\G,\prs)$ semi-symmetric  if its curvature is semi-symmetric.     
    
     Let $(\G,\prs)$ be a semi-symmetric Lorentzian Lie algebra. According to Proposition \ref{pr2}, $\G$ splits orthogonally as
     \begin{equation}\label{split}\G=\G_0\oplus \G_1\oplus\ldots\oplus \G_r,
     \end{equation}where $\G_0=\ker(\mathrm{Ric}^2)$ and $\G_1,\ldots,\G_r$ are the eigenspaces associated to the  non zero eigenvalues of $\mathrm{Ric}$. Moreover, $\Ku(\G_i,\G_j)=0$ for any $i\not=j$ and $\dim\G_i\geq2$ if $i\not=0$. The following proposition gives  more properties of the $\G_i$'s involving the Levi-Civita product.
     
     \begin{pr}\label{pr7} Let $(\G,\prs)$ be a semi-symmetric Lorentzian Lie algebra. Then, for any $i,j\in\{1,\ldots,r\}$ and $i\not=j$,
     \begin{equation*}\label{l}
             {\G_j}.\G_i\subset \G_i, \;  {\G_i}.\G_i\subset \G_0+\G_i, \;
             {\G_0}.\G_i\subset \G_i, \;
              {\G_0}.\G_0\subset \G_0, \;  {\G_i}.\G_0\subset \G_0+\G_i.
         \end{equation*}Moreover, if $\dim\G_0=1$ then for any $u\in\G_0$, $u.u=0$ and, for any $k\in\N^*$, $[\mathrm{R}_u^k,\mathrm{L}_u]=k\mathrm{R}_u^{k+1}$. In particular, $\mathrm{R}_u$ is a nilpotent endomorphism.
     
     \end{pr}
     
     \begin{proof} We start by proving that, for any $i\in\{1,\ldots,r  \}$ and any $x\in\G_i^\perp$, $\mathrm{L}_x\G_i\subset\G_i$. Fix $i\in\{1,\ldots,r  \}$ and $x\in\G_i^\perp$. For any $u,v,w\in\G_i$, by using the differential Bianchi identity, we get
         \begin{eqnarray*}
         \mathrm{L}_x(\Ku)(u,v,w)&=&-\mathrm{L}_u(\Ku)(v,x,w)-\mathrm{L}_v(\Ku)(x,u,w)\\
         &=&-\mathrm{L}_u(\Ku(v,x)w)+\Ku(\mathrm{L}_uv,x)w+\Ku(v,\mathrm{L}_ux)w+\Ku(v,x)\mathrm{L}_uw
         \\&&-\mathrm{L}_v(\Ku(u,x)w)+\Ku(\mathrm{L}_vu,x)w+\Ku(u,\mathrm{L}_vx)w+\Ku(u,x)\mathrm{L}_vw\\
         &=&\Ku(\mathrm{L}_uv,x)w+\Ku(v,\mathrm{L}_ux)w+\Ku(\mathrm{L}_vu,x)w+\Ku(u,\mathrm{L}_vx)w,
         \end{eqnarray*}since, by virtue of Proposition \ref{pr1}, $\Ku(u,x)=\Ku(v,x)=0$.
         This shows, also according to Proposition \ref{pr1}, that $\mathrm{L}_x(\Ku)(u,v,w)\in\G_i$. Now
         \begin{eqnarray*}
         \mathrm{L}_x(\Ku)(u,v,w)&=&\mathrm{L}_x(\Ku(u,v)w)-\Ku(\mathrm{L}_xu,v)w-\Ku(u,\mathrm{L}_xv)w-\Ku(u,v)\mathrm{L}_xw\\
         &=&\mathrm{L}_x(\Ku(u,v)w)-\Ku(\mathrm{L}_xu,v)w-\Ku(u,\mathrm{L}_xv)w+\Ku(v,\mathrm{L}_xw)u
         +\Ku(\mathrm{L}_xw,u)v.
         \end{eqnarray*}Since $\mathrm{L}_x(\Ku)(u,v,w)\in\G_i$ and $\Ku(\G,\G)\G_i\subset\G_i$, we get $\mathrm{L}_x(\Ku(u,v)w)\in\G_i$. Having this property in mind, we will prove now that $\mathrm{L}_x\Ri(u)\in\G_i$.  Choose an orthonormal basis $(e_1,\ldots,e_n)$  which is adapted to the splitting \eqref{split} and put $\e_i=\langle e_i,e_i\rangle$. For any $z\in\G_i^\perp$, we have
         \begin{eqnarray*}
         \langle \mathrm{L}_x\Ri(u),z\rangle&=&-\langle \Ri(u),\mathrm{L}_xz\rangle
         =\sum_{k=1}^n\e_i\langle\Ku(u,e_k)e_k,\mathrm{L}_xz\rangle
         =-\sum_{k=1}^n\e_i\langle\mathrm{L}_x(\langle\Ku(u,e_k)e_k),z\rangle
         =0.
         \end{eqnarray*}We have used the fact that if $e_k\in\G_i$ then  $\mathrm{L}_x(\langle\Ku(u,e_k)e_k)\in\G_i$ and if $e_k\in\G_i^\perp$ then $\Ku(u,e_k)=0$. Thus $\mathrm{L}_x\Ri(u)=\la_i\mathrm{L}_xu\in\G_i$ where $\la_i$ is the eigenvalues of $\Ri$ associated to the eigenspace $\G_i$. We conclude that $\mathrm{L}_x\G_i\subset\G_i$ which shows that, for any $i,j\in\{1,\ldots,r \}$ with $i\not=j$, $\mathrm{L}_{\G_j}\G_i\subset\G_i$ and $\mathrm{L}_{\G_0}\G_i\subset\G_i$. Since $\mathrm{L}$ takes its values in $\mathrm{so}(\G)$, the other inclusions follow immediately. 
         
         Suppose that $\dim\G_0=1$ an choose a non null vector $u\in\G_0$. Since $\dim\G_0=1$, $\G_0$ is nondegenerate and $\G_0.\G_0\subset\G_0$ we get $u.u=0$. Moreover, $\Ku(u,.)=0$ and hence from \eqref{cu2} $[\mathrm{R}_u,\mathrm{L}_u]=\mathrm{R}_u^{2}$. By induction, we deduce that, for any $k\in\N^*$, $[\mathrm{R}_u^k,\mathrm{L}_u]=k\mathrm{R}_u^{k+1}$. This implies that $\tr(\mathrm{R}_u^k)=0$ for any $k\geq2$ and hence $\mathrm{R}_u$ is a nilpotent endomorphism.
         \end{proof}
         
         We end this section by a lemma which we will use later.
         \begin{Le}\label{Le} Let $V$ be a pseudo-Euclidean vector space of dimension $\leq3$ and $A,B$ are, respectively, an  endomorphism and a skew-symmetric endomorphism such that $[A,B]=A^2$. Then $A=0$ or $B=0$.
              
              \end{Le}
              
              \begin{proof} The relation $[A,B]=A^2$ implies that, for any $k\in\N^*$, $[A^k,B]=kA^{k+1}$ and  $\tr(A^k)=0$ for $k\geq2$ which implies that $A$ is nilpotent. If $\dim V=2$ we have $[A,B]=0$ and if $\dim V=3$ we have $[A^2,B]=0$. To conclude it suffices to show that in a pseudo-Euclidean vector space of dimension $\leq3$ if $N$ and $B$ are, respectively,  nilpotent and   skew-symmetric satisfying $[N,B]=0$ then $B=0$ or $N=0$. Suppose $N\not=0$ and denote by $N^c$ and $B^c$ the associated complex endomorphisms of $V\otimes\C$.\\
               If $\dim V=2$ and since $[N,B]=0$ then there exists a basis of $V\otimes\C$ such that $$[N^{c}]=\left(\begin{array}{cc}0&1\\0&0\end{array} \right),\; [B^{c}]=
              \left(\begin{array}{cc}\al&0\\0&\be\end{array} \right),\;  \{\al,\be\}=\{\imath a,-\imath a \}\;\mbox{or}\; \{\al,\be\}=\{ a,- a \}.$$ The condition $[N,B]=0$ implies $a=0$ and hence $B=0$.\\
              If $\dim V=3$ and since $[N,B]=0$ then there exists a basis of $V\otimes\C$ such that $$[N^{c}]=\left(\begin{array}{ccc}0&1&0\\0&0&1\\0&0&0\end{array} \right)\;\mbox{or}\;\left(\begin{array}{ccc}0&0&0\\0&0&1\\0&0&0\end{array} \right)\;, [B^{c}]=
                   \left(\begin{array}{ccc}\al&0&0\\0&\be&0\\0&0&0\end{array} \right),\;  \{\al,\be\}=\{\imath a,-\imath a \}\;\mbox{or}\; \{\al,\be\}=\{ a,- a \}.$$
              The condition $[N,B]=0$ implies $a=0$ and hence $B=0$.
              \end{proof}

      \section{Semi-symmetric Lorentzian Lie algebras of dimension $\leq$ 4} \label{section5}

      Any 2-dimensional pseudo-Euclidean Lie algebra is semi-symmetric and we have the following known result which we will use later.
      
       \begin{pr}\label{pr8} Let $(\G,\prs)$ be a two dimensional pseudo-Euclidean Lie algebra. Then:
             \begin{enumerate}\item If $\G$ is abelian then the Levi-Civita product is trivial.
             \item If $\G$ is non abelian and $\prs$ is Euclidean then if $(e,f)$ is an orthonormal basis with $e\in[\G,\G]$ and $f\in[\G,\G]^\perp$ then
             $\mathrm{L}_f=0$, $\mathrm{L}_{e}=\la e\wedge f$ and $\Ku(e,f)=\la^2e\wedge f$.
            \item If $\G$ is non abelian,  $\prs$ is Lorentzian and $[\G,\G]$ is nondegenerate positive then if $(e,f)$ is an orthonormal basis with $e\in[\G,\G]$ and $f\in[\G,\G]^\perp$ then
                   $\mathrm{L}_f=0$, $\mathrm{L}_{e}=\la e\wedge f$ and $\Ku(e,f)=-\la^2e\wedge f$.
                   \item If $\G$ is non abelian,  $\prs$ is Lorentzian and $[\G,\G]$ is nondegenerate negative then if $(e,f)$ is an orthonormal basis with $e\in[\G,\G]$ and $f\in[\G,\G]^\perp$ then
                               $\mathrm{L}_f=0$, $\mathrm{L}_{e}=\la e\wedge f$ and $\Ku(e,f)=\la^2e\wedge f$.
               \item If $\G$ is non abelian,  $\prs$ is Lorentzian and $[\G,\G]$ is totally isotropic then if $(e,f)$ is an  basis with $e\in[\G,\G]$ and $f\in\G$ such that $\langle e,e\rangle=\langle f,f\rangle=0$, $\langle e,f\rangle=1$ then $\mathrm{L}_f=\la e\wedge f$, $\mathrm{L}_{e}=0$ and $\Ku(e,f)=0$.

             \end{enumerate}

             \end{pr}

       Let $(\G,\prs)$ be a semi-symmetric Lorentzian Lie algebra with $\dim\G=3$ or $\dim\G=4$. For any $\la\not=0$, we denote $\G_\la=\ker(\Ri-\la \mathrm{Id}_\G)$ and $\G_0=\ker(\Ri^2)$.   According to \eqref{split} and Proposition \ref{pr7},  $\G$ has one of the following types:
       \begin{enumerate}\item[]$(S30\la)$ $\dim\G=3$ and $\G=\G_0\oplus\G_\la$ with $\dim\G_0=1$, $\G_0.\G_\la\subset\G_\la$, $\G_0.\G_0\subset\G_0$ and $\la\not=0$. 
       \item[]$(S3\la)$ $\dim\G=3$ and $\G=\G_\la$ with $\la\not=0$.
       \item[]$(S30)$ $\dim\G=3$ and $\G=\G_0$.
      \item[]$(S4\mu\la)$  $\dim\G=4$ and $\G=\G_\mu\oplus\G_\la$ with $\dim\G_\mu=\dim\G_\la=2$,  $\la\not=\mu$,  $\la\not=0$, $\mu\not=0$, $\G_\mu.\G_\la\subset\G_\la$, $\G_\la.\G_\mu\subset\G_\mu$, $\G_\la.\G_\la\subset\G_\la$ and $\G_\mu.\G_\mu\subset\G_\mu$.
      \item[]$(S40^1\la)$ $\dim\G=4$ and $\G=\G_0\oplus\G_\la$ with $\dim\G_0=1$, $\G_0.\G_\la\subset\G_\la$, $\G_0.\G_0\subset\G_0$ and $\la\not=0$.
      \item[]$(S40^2\la)$  $\dim\G=4$ and $\G=\G_0\oplus\G_\la$ with $\dim\G_\la=2$, $\G_0.\G_\la\subset\G_\la$,  $\G_0.\G_0\subset\G_0$ and $\la\not=0$.
       \item[]$(S4\la)$ $\dim\G=4$ and $\G=\G_\la$ with $\la\not=0$.
             \item[]$(S40)$ $\dim\G=4$ and $\G=\G_0$.
            \end{enumerate} Moreover, for any type above we have from Sections \ref{section2} and \ref{section3}  an exact expression of the curvature. Therefore, in order to complete the study, we need to determine from the curvature the Levi-Civita product and hence the Lie algebra structure. This is the purpose of what will follow.

      \subsection{ Type $(S30\la)$} 
      
      \begin{pr}\label{30la} Let $(\G,\prs)$ be a three dimensional semi-symmetric Lorentzian Lie algebra of type $(S30\la)$. Then $\G_0.\G=\G.\G_0=0$, $\G_\la$ is an ideal for the Levi-Civita product and hence $\G$ is a product as a Lie algebra of $\G_0$ with $\G_\la$.
      
      \end{pr}
    \begin{proof}  
       We have $\G=\G_0\oplus\G_\la$ with $\dim\G_0=1$,  $\la\not=0$ and $\G_0.\G_\la\subset\G_\la$ and $\G_0.\G_0=\{0\}$. This implies that $\G_\la.\G_0\subset\G_\la$. Choose a generator $u$ of $\G_0$. According to Proposition \ref{pr7}, $\mathrm{R}_u$ is nilpotent and $[\mathrm{R}_u,\mathrm{L}_u]=\mathrm{R}_u^2$. But $\mathrm{R}_u(u)=0$ and $\mathrm{R}_u(\G_\la)\subset\G_\la$  hence $\mathrm{R}_u^2=0$.
Moreover,       
        by virtue of Propositions \ref{pr2}, \ref{pr3} and \ref{pr7}, for any $v,w\in\G_\la$, $\Ku(v,w)=-\frac{\la}2 v\wedge w$ and $\Ku(u,.).=\Ku(.,.)u=0$.  To conclude, we will  
             show that $\mathrm{L}_u=\mathrm{R}_u=0$. Remark that $\mathrm{R}_u=0$ is equivalent to $\G_\la.\G_\la\subset\G_\la$.\\
          According to Lemma \ref{Le}, $\mathrm{L}_u=0$ or $\mathrm{R}_u=0$.   Suppose first that $\mathrm{R}_u=0$.  Then $\G_\la$ is a pseudo-Euclidean Lie algebra with non vanishing curvature and $\mathrm{L}_u$ is a skew-symmetric derivation of $\G_\la$ and hence $\mathrm{L}_u=0$.
                         Suppose now that   $\mathrm{R}_u\not=0$. Then $\mathrm{L}_u=0$ and $\mathrm{Im}\mathrm{R}_u$ is a one dimensional subspace of $\G_\la$. Choose a generator $v=x.u\in\mathrm{Im}\mathrm{R}_u$. We have
          $$0=\mathrm{L}_{[u,x]}-[\mathrm{L}_u,\mathrm{L}_x] =\mathrm{L}_{x.u}.$$So $\mathrm{L}_v=0$. Then, for any $w\in\G_\la$,
          $$-\frac{\la}2 v\wedge w=\mathrm{L}_{[v,w]}-[\mathrm{L}_v,\mathrm{L}_w] =\mathrm{L}_{w.v}.$$
          Write $w.v=au+w_1$ with $w_1\in\G_\la$. We have $w_1\not=0$,  $\langle w_1,v\rangle=0$ and $\{w_1,v\}$ are linearly independent. This implies that $\G_\la.\G_\la\subset\G_\la$ and $\mathrm{R}_u=0$, which completes the proof.\end{proof}    
      
 \subsection{ Type $(S30)$} 
 
 In this case $\G=\G_0=\ker\Ri^2$. If $\Ri=0$ then $\G$ is a flat Lorentzian Lie algebra. There are, up to an isomorphism, six three dimensional flat Lorentzian Lie algebras, three unimodular and three nonunimodular (see \cite{bou1, bou}).

 \begin{pr}\label{pr30} Let $(\G,\prs)$ be a three dimensional semi-symmetric Lorentzian Lie algebra such that its curvature is Ricci isotropic. Then there exists a basis $(e,f,g)$ of $\G$ such that the non vanishing products are $\langle e,e\rangle=\langle f,g\rangle=1$, $\langle f,f\rangle=\langle g,g\rangle=0$ and the non vanishing Lie brackets have one of the following types:\begin{enumerate}\item[$(i)$]
 $ [e,f]=af,\;[e,g]=be-ag+\frac{1+2b^2}{2a}f,\;[f,g]=-bf,a,b\in\R, a\not=0.$\item[$(ii)$]
             $[e,g]=ae+bf,\;[f,g]=\frac{1+a^2}{a}f,\; a,b\in\R,\;a\not=0.$\end{enumerate}
      In both cases, $\G$ is not second-order locally symmetric and $\mathfrak{h}(\G)=\mathfrak{h}(\Ku)=\mathrm{span}\{e\wedge f\}$.

 \end{pr} \begin{proof}  
      According to Proposition \ref{pr3}, there exists
     a basis $(e,f,g)$ of $\G$ such that the non vanishing products are $\langle e,e\rangle=\langle g,f\rangle=1$, 
         $\Ku(e,f)=\Ku(f,g)=0$ and $\Ku(e,g)=-e\we f$. 
         Put
         \[\mathrm{L}_e=ae\wedge f+be\wedge g+cf\wedge g,\;\mathrm{L}_f=xe\wedge f+ye\wedge g+zf\wedge g\esp \mathrm{L}_g=pe\wedge f+qe\wedge g+rf\wedge g.   \]
          We have
         \begin{eqnarray*}
         \mathrm{L}_e(\Ku)(f,g)&=&-\Ku(\mathrm{L}_ef,g)-\Ku(f,\mathrm{L}_eg)
         =be\wedge f,\\
         \mathrm{L}_f(\Ku)(g,e)&=&[\mathrm{L}_f,e\wedge f]-\Ku(\mathrm{L}_fg,e)-\Ku(g,\mathrm{L}_fe)
         =(\mathrm{L}_fe)\wedge f+e\wedge \mathrm{L}_ff+z\Ku(g,e)
         =-y g\wedge f+2ze\wedge f,\\
         \mathrm{L}_g(\Ku)(e,f)&=&-\Ku(\mathrm{L}_ge,f)-\Ku(e,\mathrm{L}_gf)
         =0.
         \end{eqnarray*}
         So the differential Bianchi identity gives $y=0$ and $b=-2z$. On the other hand, the relation
         $ 0=\mathrm{L}_{[e,f]}-[\mathrm{L}_{e},\mathrm{L}_{f}]$
        is equivalent to
        $ z^2=x^2-az=3xz-cz=0$ and hence $z=y=b=x=0$. Now the relations 
        $-e\wedge f=\mathrm{L}_{[e,g]}-[\mathrm{L}_{e},\mathrm{L}_{g}]$ and $\mathrm{L}_{[e,g]}-[\mathrm{L}_{e},\mathrm{L}_{g}]=0$ are equivalent to
      
      \[ q^2=a^2+1-2pc+pq+ar=ac-rc+rq+aq=aq=qc=0. \]  Thus $q=0$ and $c(a-r)=
      a^2-2cp+1+ar=0$. Therefore, the solutions are
      \[ (x=y=z=b=c=q=0\esp a^2+ar+1=0)\quad\mbox{or}\quad(x=y=z=b=q=0,c\not=0, a=r\esp p=\frac{2r^2+1}{2c}). \]Hence
      \[\mathrm{L}_e=ae\wedge f+cf\wedge g,\;\mathrm{L}_f=0\esp \mathrm{L}_g=pe\wedge f+rf\wedge g,   \]
      where $(c=0, a^2+ar+1=0)$ or $(c\not=0,p=\frac{2r^2+1}{2c})$. In both cases it is easy to check that $\mathfrak{h}(\Ku)=\mathrm{span}\{e\wedge f  \}$ is invariant by $\mathrm{L}$ which shows that it is the holonomy Lie algebra. Moreover, for the first case we have $\mathrm{L}^2_{g,g}(\Ku)(e,g)=-\frac{6(1+a^2)^2}{a^2} e\wedge f$ and in the second case $\mathrm{L}^2_{e,g}(\Ku)(e,g)=-4rc e\wedge f$ and $\mathrm{L}^2_{g,g}(\Ku)(e,g)=(1-4r^2) e\wedge f$. This shows that in both cases $\G$ is not second-order locally symmetric.
      \end{proof}

  \subsection{ Type $(S3\la)$} 
  
  In this case, $\G$ is a three dimensional  pseudo-Euclidean Lie algebra of constant curvature. If $\G$ is unimodular Euclidean then $\G$ is isometric to $\mathrm{so}(3)$ endowed with a multiple of the Killing form. If $\G$ is unimodular Lorentzian then $\G$ is isometric to $\mathrm{so}(2,1)$ endowed with a multiple of the Killing form. The situation in nonunimodular case is more complicated. Before giving a precise description of such pseudo-Euclidean Lie algebras, we introduce, for any nonunimodular pseudo-Euclidean Lie algebras $\G$, the vector $\mathbf{h}$ defined by $\langle u,\mathbf{h}\rangle=\tr(\ad_u)$. We have obviously, $\mathbf{h}.\mathbf{h}=0$ and since $\mathbf{h}\in[\G,\G]^\perp$, $\mathrm{R}_{\mathbf{h}}$ is a symmetric endomorphism. The proof of the following theorem is long and it is non relevant for our study so we omit it.
  
  \begin{theo} Let $(\G,\prs)$ be a three dimensional pseudo-Euclidean nonunimodular Lie algebra of constant curvature with  $\Ri=\la\mathrm{Id}_\G$. Then one of the following situations holds:
  \begin{enumerate}
  \item $\G$ is Euclidean   then there exists an orthonormal basis $(e,f,g)$ such that the non vanishing Lie brackets are given by
      \[ [e,f]=-cg+\sqrt{-\frac\la2}f,\;[e,g]=cf+\sqrt{-\frac\la2}g. \]                   
    \item $\G$ is Lorentzian   with $\langle \mathbf{h},\mathbf{h}\rangle<0$ then there exists an orthonormal basis $(e,f,g)$ with $e=\left(-\langle \mathbf{h},\mathbf{h}\rangle \right)^{-\frac12}\mathbf{h}$ such that the non vanishing Lie brackets are given by
              \[ [e,f]=-cg+\sqrt{\frac\la2}f,\;[e,g]=cf+\sqrt{\frac\la2}g. \] 
    \item $\G$ is Lorentzian   with $\langle \mathbf{h},\mathbf{h}\rangle>0$ and $\mathrm{R}_\mathbf{h}$ is diagonalizable then there exists an orthonormal basis $(e,f,g)$ with $e=\left(\langle \mathbf{h},\mathbf{h}\rangle \right)^{-\frac12}\mathbf{h}$, $\langle g,g\rangle=-1$ and such that the non vanishing Lie brackets are given by  
                    \[ [e,f]=-cg+\sqrt{-\frac\la2}f,\;[e,g]=-cf+\sqrt{-\frac\la2}g. \] 
  \item   $\G$ is Lorentzian   with $\langle \mathbf{h},\mathbf{h}\rangle>0$  and $\mathrm{R}_\mathbf{h}$ is non diagonalizable then there exists an orthonormal basis $(e,f,g)$ with $e=\left(\langle \mathbf{h},\mathbf{h}\rangle \right)^{-\frac12}\mathbf{h}$, $\langle g,g\rangle=\langle f,f\rangle=0$, $\langle f,g\rangle=1$ and such that the non vanishing Lie brackets are given by  
                  \[ [e,f]=2\sqrt{-\frac\la2}f,\;[e,g]=af\quad a<0. \]

  \end{enumerate}

  \end{theo}

  \subsection{ Type  $(S4\mu\la)$}
  
  \begin{pr} Let $(\G,\prs)$ be a four-dimensional semi-symmetric Lorentzian Lie algebra of type $(S4\mu\la)$. Then $\G_\la.\G_\mu=\G_\mu.\G_\la=0$ and hence $\G$ is the product of a two dimensional Euclidean Lie algebra with a two dimensional Lorentzian Lie algebra.
  
  \end{pr}
  
  \begin{proof}  We have    $\G=\G_\mu\oplus\G_\la$ with $\mu\not=0$, $\la\not=0$,  $\mu\not=\la$   $\G_\mu.\G_\mu\subset\G_\mu$,  $\G_\la.\G_\la\subset\G_\la$, $\G_\la.\G_\mu\subset\G_\mu$ and $\G_\mu.\G_\la\subset\G_\la$. We can suppose that $\G_\mu$ is Euclidean and $\G_\la$ is Lorentzian. According to Proposition \ref{pr7}, there exists an orthonormal basis $(e,f)$ of $\G_\mu$ and an orthonormal basis $(g,h)$ of $\G_\la$ such that, in restriction to $\G_\mu$, $\mathrm{L}_f$ vanishes and, in restriction to $\G_\la$, $\mathrm{L}_h$ vanishes. So
  \[ \mathrm{L}_{e}=ae\wedge f+bg\wedge h,\;\mathrm{L}_{f}=dg\wedge h,\;
  \mathrm{L}_{g}=ue\wedge f+vg\wedge h,\; \mathrm{L}_{h}=p e\wedge f,\;\Ku(e,f)=-\la e\wedge f\esp \Ku(g,h)=-\mu g\wedge h. \]

                     We have
                     \[ [e,f]=ae,\;[e,g]=bh+uf,\;[e,h]=bg+pf,\;[f,g]=dh-ue,\;[f,h]=dg-pe,\;[g,h]=vg. \]
      The relations
      \[ -\mu e\wedge f=\mathrm{L}_{[e,f]}-[\mathrm{L}_e,\mathrm{L}_f]\esp
      -\la g\wedge h=\mathrm{L}_{[g,h]}-[\mathrm{L}_g,\mathrm{L}_h] \]are equivalent to $a^2=-\la$, $v^2=\mu$, $ab=vu=0$ and hence $u=b=0$. Now the relation $0= \mathrm{L}_{[f,h]}-[\mathrm{L}_f,\mathrm{L}_h]$ is equivalent to $ap=dv-bp=0$ and hence $p=d=0$ and we get the result.              
        \end{proof}

  \subsection{ Type $(S40^1\la)$} 
  
  \begin{pr} Let $(\G,\prs)$ be a four-dimensional semi-symmetric Lorentzian Lie algebra of type $(S40^1\la)$. Then $\G.\G_0=0$, $\G_\la.\G_\la\subset\G_\la$ and hence $\G$ the semi-direct product of  $\G_0$  with the three dimensional pseudo-Euclidean Lie algebra $\G_\la$ of constant curvature and the action of $\G_0$ on $\G_\la$ is by a skew-symmetric derivation.
    
    \end{pr}
    
    \begin{proof}  We have $\G=\G_0\oplus\G_\la$ with $\dim\G_0=1$,  $\la\not=0$ and $\G_0.\G_\la\subset\G_\la$ and $\G_0.\G_0=\{0\}$. This implies that $\G_\la.\G_0\subset\G_\la$. Choose a generator $u$ of $\G_0$. According to Proposition \ref{pr7}, $\mathrm{R}_u$ is nilpotent and $[\mathrm{R}_u,\mathrm{L}_u]=\mathrm{R}_u^2$. But $\mathrm{R}_u(u)=0$ and $\mathrm{R}_u(\G_\la)\subset\G_\la$ and hence, according to Lemma \ref{Le}, $\mathrm{R}_u=0$ or $\mathrm{L}_u=0$. Moreover,       
                by virtue of Propositions \ref{pr2}, \ref{pr3} and \ref{pr7}, for any $v,w\in\G_\la$, $\Ku(v,w)=-\frac{\la}2 v\wedge w$ and $\Ku(u,.).=\Ku(.,.)u=0$.  Let show that $\mathrm{R}_u=0$.

      Suppose  that $\mathrm{R}_u\not=0$,  hence $\mathrm{L}_u=0$ and $\mathrm{R}_u^2=0$. Then   $\mathrm{Im}\mathrm{R}_u$ is a one dimensional subspace of $\G_\la$. Choose a generator $v=x.u\in\mathrm{Im}\mathrm{R}_u$. We have, 
                $$0=\mathrm{L}_{[u,x]}-[\mathrm{L}_u,\mathrm{L}_x] =\mathrm{L}_{x.u}.$$So $\mathrm{L}_v=0$. Then, for any $w\in\G_\la$,
                $$-\frac{\la}2 v\wedge w=\mathrm{L}_{[v,w]}-[\mathrm{L}_v,\mathrm{L}_w] =\mathrm{L}_{w.v}.$$ Consider $\mathrm{R}_v:\G_\la\too\G$. From the relation above, we have $\ker\mathrm{R}_v=\R v$. So there exists two linearly independent vectors $v_1,v_2\in\G_\la$ such that $\{v,v_1,v_2\}$ is a basis of $\G_\la$,  $\{v_1.v,v_2.v\}$ are linearly independent with
                \[ \mathrm{L}_{v_1.v}=-\la v\wedge v_1\esp \mathrm{L}_{v_2.v}=-\la v\wedge v_2.\]This implies that $\G_\la.\G_\la\subset \G_\la$ and hence $\mathrm{R}_u=0$.  Finally, $\mathrm{R}_u=0$.
                
                Now $D=\mathrm{L}_u=\ad_u$ is a skew-symmetric derivation of $\G_\la$. If $\G_\la$ is unimodular then $D=\ad_v$ with $v\in\G_\la$ and since the metric on $\G_\la$ is bi-invariant, for any $w\in\G_\la$, $\mathrm{L}_w=\frac12\ad_w$. So
                \[ \Ku(u,w)=\mathrm{L}_{[u,w]}-[\mathrm{L}_u,\mathrm{L}_w]=\frac12\ad_{[v,w]}-\frac12[\ad_v,\ad_w]=0. \]
                
                If $\G_\la$ is nonunimodular then, for any $v\in\G_\la$ $\ad_{Dv}=[D,\ad_v]$ and hence $0=\tr(\ad_{Dv})=\langle Dv,\mathbf{h}\rangle$ which implies that $D\mathbf{h}=0$. One can check easily that this condition  suffices to insure that $\Ku(u,v)=0$ for any $v\in\G_\la$.\end{proof}

    \subsection{ Type $(S40^2\la)$}
    
    \begin{pr} Let $(\G,\prs)$ be a four-dimensional semi-symmetric Lorentzian Lie algebra of type $(S40^2\la)$. Then  $\G_0.\G=0$, $\G_\la.\G_\la\subset\G_\la$, $\G_\la.\G_0\subset\G_0$ and hence $\G$ is the semi-direct product of the pseudo-Euclidean Lie algebra $\G_\la$  with the abelian Lie algebra $\G_0$ and the action of $\G_\la$ on $\G_0$ is given by skew-symmetric endomorphisms.
        \end{pr}
      \begin{proof}  We have $\G=\G_0\oplus\G_\la$ with $\dim\G_0=2$, $\G_0.\G_0\subset\G_0$ and $\G_0.\G_\la\subset\G_\la$. Moreover, for any $u\in\G_0$ and $v,w\in\G_\la$, $\Ku(u,.).=\Ku(.,.)u=0$ and $\Ku(v,w)=-\frac{\la}2 u\wedge v$.       Let first show that $\G_0.\G_0=\{0\}$. Since $\G_0$ is a pseudo-Euclidean Lie algebra with vanishing curvature then $\G_0.\G_0=\{0\}$   when  $\G_0$ is Euclidean. If $\G_0$ is Lorentzian then according to Proposition \ref{pr8}, there exists a basis $(e,f)$ of $\G_0$ with $\langle e,f\rangle=1$ such that
            \[ \mathrm{L}_e=a g\wedge h,\; \mathrm{L}_f=ce\wedge f+b g\wedge h\esp[e,f]=cf. \]
      But the Lie algebra of skew-symmetric endomorphisms of a 2-dimensional pseudo-Euclidean vector space is abelian then, for any $u,v\in\G_0$, we have $[\mathrm{L}_u,\mathrm{L}_v]=0$ and hence $\mathrm{L}_{[u,v]}=0$.     Thus $c\mathrm{L}_f=0$ which implies $\G_0.\G_0=\{0\}$.\\
      Consider $N=\{u\in\G_0,\mathrm{L}_u=0  \}$. Since $\G_0.\G_0=\{0\}$ and $\dim \mathrm{L}(\G_0)\leq 1$ we have $\dim N\geq1$.  Suppose that $\dim N=1$. Therefore, we can choose an orthonormal basis $(e,f)$ of $\G_0$ such that $\mathrm{L}_e\not=0$ and $\mathrm{L}_f\not=0$. Since $e.e=0$, $\mathrm{L}_e$ left invariant $e^\perp$. We have also $\langle \mathrm{R}_ev,e\rangle=0$ and hence $\mathrm{R}_e$ leaves invariant $e^\perp$. Since $e.e=0$, we get from \eqref{cu2} that $[\mathrm{R}_e,\mathrm{L}_e]=\mathrm{R}_e^2$. According
      to Lemma \ref{Le}, the restriction of $\mathrm{R}_e$ to $e^\perp$ vanishes and hence its vanishes. A same argument shows that $\mathrm{R}_f=0$ and hence for any $u\in\G_0$, $\mathrm{R}_u=0$. This implies that $\G_\la.\G_\la\subset\G_\la$. Now, for any $u\in\G_0$, $\mathrm{L}_u$ is a skew-symmetric derivation of $\G_\la$ and hence $\mathrm{L}_u=0$. So we have shown that, for any $u\in\G_0$, $\mathrm{L}_u=0$. Let show now that $\G_\la.\G_\la\subset \G_\la$. Remark first that is equivalent to $\mathrm{Im}\mathrm{R}_u\subset\G_0$ for any $u\in\G_0$.\\
      Suppose that there exists $u\in\G_0$ such that $\mathrm{Im}\mathrm{R}_u\nsubset\G_0$. This means that there exists $v\in\G_\la$ such that $v.u=v_0+v_1$ where $v_0\in\G_0$ and $v_1\in\G_\la$ with $v_1\not=0$. Then $\mathrm{L}_{v.u}=\mathrm{L_{[v,u]}}=\mathrm{L}_{v_1}=0$. Therefore, for any $w\in\G_\la$, $\mathrm{L}_{w.v}=-\frac{\la}2 w\wedge v$. This implies that $\G_\la.\G_\la\subset \G_\la$ which is a contradiction. So we have proved so far that, for any $u\in\G_0$, $\mathrm{L}_u=0$,  $\G_\la.\G_\la\subset \G_\la$ and $\G_\la.\G_0\subset \G_0$. So $\G$ is the semi-direct product of $\G_\la$ with $\G_0$ and the action of $\G_\la$ on $\G_0$ is given by skew-symmetric endomorphisms.
       \end{proof}
    
    \subsection{ Type $(S4\la)$}
    
    In this case $\G$ is Einstein locally symmetric either of non null constant curvature if its primitive holonomy algebra has dimension 6 or not of constant curvature if its primitive holonomy algebra is of dimension 2.  In \cite{calvaruso}, there is a classification of four-dimensional Lorentzian Lie algebras, based on this classification we get the following  result.
    
    \begin{theo} Let $(\G,\prs)$ be a four-dimensional semi-symmetric Lorentzian Lie algebra of type $(S4\la)$. Then $\G$ is isomorphic $\R^4$ with its canonical basis $(e_i)_{i=1}^4$ and $\langle e_3,e_3\rangle=-1$ and the non vanishing Lie brackets has  one of the following forms:
    \begin{enumerate}\item $[e_1,e_2]=\e ae_1,[e_1,e_3]=ae_1,[e_1,e_4]=\de ae_1,[e_3,e_4]=-2a\de(\e e_2-e_3)$ $(a\not=0)$, (Constant sectional curvature $-a^2$),
    \item $[e_1,e_2]= \frac{\e \sqrt{a^2-b^2}}{2}e_1,[e_1,e_3]=-\frac{\de\e \sqrt{a^2-b^2}}{2}e_1,[e_1,e_4]=\frac{\de a+b}{2}e_1,[e_2,e_4]=b(e_2+\de e_3),[e_3,e_4]=a( e_2+\de e_3)$ $(b\not=-\de a)$, (Constant sectional curvature $-\frac{(a+\de b)^2}{4}$),
    \item $[e_1,e_2]= \frac{\e a\sqrt{a^2-b^2}}{b}e_1,[e_1,e_3]=\e \sqrt{a^2-b^2}e_1,[e_2,e_4]=be_2-a e_3,[e_3,e_4]=a e_2-\frac{a^2}b e_3$ $(b\not=\pm a)$, $\la=-\frac{(a^2-b^2)^2}{b^2}$, $\dim\mathfrak{h}(\Ku)=2$,
      \item $[e_1,e_2]=\e \sqrt{a^2-b^2}e_1+be_2, [e_3,e_4]=a e_3$ $(a\not=0)$, $\la=-a^2$ and $\dim\mathfrak{h}(\Ku)=2$,
      \item $[e_1,e_4]=ae_1,[e_2,e_4]=ae_2+be_3,[e_3,e_4]=be_2+ae_3$ $(a\not=0)$, (Constant sectional curvature $-a^2$),
      \item $[e_1,e_4]=\e \frac23ae_1,[e_2,e_4]=\e\frac23ae_2+ae_3,[e_3,e_4]=ae_2+\e\frac23ae_3$ $(a\not=0)$. (Constant sectional curvature $-\frac{4a^2}9$)
    
    \end{enumerate}
    In all the brackets above $\de=\pm1$ and $\e=\pm1$.
    
    \end{theo}

    \subsection{\bf Type $(S40)$}
    
    If $\Ri=0$ then $\G$ is Ricci flat and, according to Theorem \ref{main1}, its curvature tensor has a Petrov normal form of type II with $\wi\Ku^2=0$. In the list obtained by Calvaruso-Zaeim \cite{calvaruso}, the condition $\wi\Ku^2=0$ is equivalent to $\wi\Ku=0$.
    
    Before studying the others cases let make some important remarks. According to Theorem \ref{th22}, there exists a basis
        $(e,f,g,h)$  a  basis of $\G$ such that the non vanishing products are
            $\langle e,e\rangle=\langle f,f\rangle=\langle g,h\rangle=1$ and the non vanishing curvatures are 
       $\Ku(e,h)=Ae\we g$, $\Ku(f,h)=Bf\we g$
         with $A+B=-1$. Denote by $\mathfrak{h}(\G)$ the holonomy Lie algebra of $\G$. It is the smallest Lie algebra containing the $\Ku(u,v)$ and satisfying $[\mathrm{L}_u,\mathfrak{h}(\G)]\subset\mathfrak{h}(\G)$, for any $u,v\in\G$.
    The holonomy Lie algebras of Lorentzian manifolds are well understood (see \cite{galaev}) and, for our purpose, we recall some known facts. We have three situations:\begin{enumerate}\item 
     $\mathfrak{h}(\G)$ can be irreducible, i.e., $\mathfrak{h}(\G)$ leaves invariant no proper vector subspace of $\G$. In this case $\mathfrak{h}(\G)=\mathrm{so}(\G)$.
     \item $\mathfrak{h}(\G)$ can be weekly irreducible, i.e., $\mathfrak{h}(\G)$ leaves invariant no proper nondegenerate vector subspace of $\G$ but leaves invariant a degenerate vector subspace $U$. Then the null line $U\cap U^\perp=\R p$ is also invariant. Since $\Ri$ is totally isotropic,  Theorem 7 in \cite{galaev1} shows that $\mathfrak{h}(\G)\subset\mathrm{span}\{ e\wedge f,e\wedge g,f\wedge g, g\wedge h  \}$.  Moreover, the left invariant vector field associated to $p$ is recurrent and hence $\mathrm{L}_up=\theta(u)p$ for any $u\in\G$. 
     \item $\mathfrak{h}(\G)$ can be decomposable, i.e., $\G=\G_1\oplus\G_2$ where the $\G_i$ are nondegenerate invariant by $\mathfrak{h}(\G)$.
     \end{enumerate}

      Before starting the computation,  remark that if    ($A\not=0$ and $B\not=0$) then $\G$ is indecomposable. Indeed, if $E$ is a nondegenerate vector subspace invariant by $\mathfrak{h}(\G)$ then $E$ is invariant by $e\we g$ and $f\we g$ and we can suppose that $\dim E=1$ or $2$. If $\dim E=1$ then $E\subset\{e,g \}^\perp\cap\{f,g\}^\perp=\R g$ which is impossible. A same argument leads to a contradiction when $\dim E=2$.

    \begin{theo}\label{main2}
    Let $(\G,\langle\;,\; \rangle)$ be a four-dimensional semi-symmetric Lorentzian  Lie algebra of type $(S40)$ with $\Ri\neq 0$ and $\dim\mathfrak{h}(\Ku)=2$. Then, there  exists  a basis $(e,f,g,h)$ with the non vanishing products  $\langle e, e \rangle=\langle f, f \rangle=\langle g, h \rangle=1 $ and the non vanishing brackets have one of the following forms:\begin{enumerate}
    \item$ [e,f]=(a-b)g,\;[e,h]=\e\sqrt{ab+\frac12}e+(b+x)f+zg,\;[f,h]=(a-x)e+\e\sqrt{ab+\frac12}f+yg,\;[g,h]=2\e\sqrt{ab+\frac12}g, a\not=b.$ 
          \item $ [e,f]=(a-\frac{2bc-1}{2a})g,\;[e,h]=ce+\frac{2bc-1}{2a}f+zg,\;[f,h]=ae+bf+yg,\;[g,h]=(c+b)g,a-\frac{2bc-1}{2a}\not=0.$ 
               
          \item $[e,h]=ae+xf+ag,\;[f,h]=-xe+af+yg,\;[g,h]=\frac{2a^2+1}{2a}g.$  
             \item $ \;[e,h]=\e\sqrt{\frac{2a^2+1}{2}}e+(a+x)f+zg,\;[f,h]=(a-x)e+
             \e\sqrt{\frac{2a^2+1}{2}}f+yg,\;[g,h]=2\e\sqrt{\frac{2a^2+1}{2}}g.$  
          \item $ [e,h]=ce+(a+\frac{2a^3+a-2abc}{b^2-c^2})f+zg,\;[f,h]=(a-\frac{2a^3+a-2abc}{b^2-c^2})e+bf+yg,\;[g,h]=
            \frac{2a^2+b^2+c^2+1}{b+c}g.$ 
             
              \end{enumerate}     
      In all what above $\e^2=1$. Moreover,  all the models above are not second-order locally symmetric and satisfy $\mathfrak{h}(\Ku)=\mathfrak{h}(\G)$.       
    
    \end{theo}
    
    \begin{proof}
   According to Theorem \ref{th22}, there exists a basis
          $(e,f,g,h)$   of $\G$ such that the non vanishing products are
              $\langle e,e\rangle=\langle f,f\rangle=\langle g,h\rangle=1$ and the non vanishing curvatures are 
         $\Ku(e,h)=Ae\we g$, $\Ku(f,h)=Bf\we g$
           with $A+B=-1$, $A\not=0$ and $B\not=0$. Put    
      \[[\Lu_e]=\left(%
    \begin{array}{cccc}
      0 & a &u_1&c \\
      -a & 0 & u_2&l \\
      -c & -l & -k&0 \\
    	-u_1&-u_2&0&k\\
    \end{array}%
    \right),~~[\Lu_f]=\left(%
    \begin{array}{cccc}
      0 & m &v_1&d \\
      -m & 0 & v_2&q \\
      -d & -q & -r&0 \\
    	-v_1&-v_2&0&r\\
    \end{array}%
    \right),\]\[[\Lu_g]=\left(%
    \begin{array}{cccc}
      0 & s &w_1&u \\
      -s & 0 &w_2&z \\
      -u & -z & -w&0 \\
    	-w_1&-w_2&0&w\\
    \end{array}%
    \right),~~[\Lu_h]=\left(%
    \begin{array}{cccc}
      0 & x &p_1&n \\
      -x & 0 & p_2&y \\
      -n & -y & -b&0 \\
    	-p_1&-p_2&0&b\\
    \end{array}%
    \right).\]
    	The differential  Bianchi  identity gives
     \begin{align*}
     0&=\Lu_e(\Ku)(f,g)+\Lu_f(\Ku)(g,e)+\Lu_g(\Ku)(e,f)=w_2Ae\wedge g-w_1Be\wedge f,\\
    		0&=\Lu_e(\Ku)(f,h)+\Lu_f(\Ku)(h,e)+\Lu_h(\Ku)(e,f)=(a(B-A)+
    		(2r+p_2)A)e\wedge g-(m(B-A)+2Bk+p_1B)f\wedge g\\
    		&-(u_1B+v_2A)e\wedge f+(u_2B-v_1A)g\wedge h,\\
    	0&=\Lu_e(\Ku)(g,h)+\Lu_g(\Ku)(h,e)+\Lu_h(\Ku)(e,g)=(2w-u_1) Ae\wedge g+(s(A-B)-u_2B)f\wedge g-w_2Ae\wedge f+w_1Ag\wedge h,\\
    0&=\Lu_f(\Ku)(g,h)+\Lu_g(\Ku)(h,f)+\Lu_h(\Ku)(f,g)=(s(A-B)-v_1A)e\wedge g+(2w-v_2)B f\wedge g-w_1Be\wedge f+w_2Bg\wedge h.
    \end{align*}
    So $w_1=w_2=0$ and
    \begin{equation}\label{bianchi}0=u_2B-v_1A=u_1B+v_2A=(2w-u_1)A=(2w-v_2)B=s(A-B)-v_1A=s(A-B)-u_2B=a(B-A)+(2r+p_2)A=
    m(B-A)+(2k+p_1)B.
    	\end{equation}
Since $A\not=0$ and $B\not=0$ from \eqref{bianchi} we get $u_1=v_2=w=0$. On the other hand, since $g.g=0$, we get from \eqref{cu2}
$[\mathrm{R}_g,\mathrm{L}_g]=\mathrm{R}_g^2$. This implies that $\mathrm{R}_g$ is nilpotent and hence $\mathrm{R}_g^4=0$. Or, $[\mathrm{R}_g]=\left(\begin{array}{cccc}
0&v_1&0&p_1\\u_2&0&0&p_2\\-k&-r&0&-b\\0&0&0&0
\end{array}  \right)$ and  a direct computation shows that $\mathrm{R}_g^4=0$ implies $v_1u_2=0$ and from \eqref{bianchi} we get $v_1=u_2=0$. The relation $[\mathrm{R}_g,\mathrm{L}_g]=\mathrm{R}_g^2$ is equivalent to
\[ sp_1=sp_2=sr=sk=-uk-rz+kp_1+rp_2+up_1+zp_2=0. \]
We have two cases.
\begin{enumerate}\item $s\not=0$. Then $A=B=-\frac12$ and $p_1=p_2=k=r=0$. We consider the equations   	 \begin{equation}\label{eq}\left\{ \begin{array}{l}
    	\Ku(e,f)=\Lu_{[e,f]}-[\Lu_e,\Lu_f]=	a\Lu_e+m\Lu_f+(d-l)\Lu_g-[\Lu_e,\Lu_f]=0,\\
    	\Ku(e,g)=\Lu_{[e,g]}-[\Lu_e,\Lu_g]=s\Lu_f+(u-k)\Lu_g-[\Lu_e,\Lu_g]=0,\\
    	\Ku(e,h)=\Lu_{[e,h]}-[\Lu_e,\Lu_h]=c\Lu_e+(l+x)\Lu_f+n\Lu_g+k\Lu_h-[\Lu_e,\Lu_h]=A(e\we g),\\
    	\Ku(f,g)=\Lu_{[f,g]}-[\Lu_f,\Lu_g]=
    		-s\Lu_e+(z-r)\Lu_g-[\Lu_f,\Lu_g]=0,\\
    	\Ku(f,h)=\Lu_{[f,h]}-[\Lu_f,L_h]=	(d-x)\Lu_e+q\Lu_f+y\Lu_g+r\Lu_h-[\Lu_f,L_h]=B(f\we g),\\
    	\Ku(g,h)=\Lu_{[g,h]}-[\Lu_g,\Lu_h]=
    	u\Lu_e+z\Lu_f+b\Lu_g-[\Lu_g,\Lu_h]=0.\end{array}\right.
    \end{equation}From the second equation $m=-u$. From  the fourth equation we get $z=a$ and from the  sixth equation we get $b=0$. The equations become
            \[ \left\{\begin{array}{ccc}a^2+u^2+sd-sl&=&0,\\ 
            ac-2ul-aq&=&0,\\-uq+2ad+cu&=&0,\\sl-a^2+u^2+sd&=&0,\\-cs+2au+sq&=&0,\\ 
            ac-ul-ux+ns&=&0,\\ c^2+dl+dx+nu+lx-ay-A&=&0,\\
            cl+lq+qx+2an-cx&=&0,\\ad-ax-uq+ys&=&0,\\cd-cx+qd+2yu+qx&=&0,\\
            dl-lx+q^2+ay-dx-nu-B&=&0,\\cu+ad+ax-ys&=&0,\\ ul+aq-ux+ns&=&0. \end{array} \right. \] Then $u^2=-sd$,  $a^2=sl$ and $c-q=2s^{-1}au$.  So
            \[ c^2+q^2+2dl+1=(c-q)^2+2qc+2dl+1=4s^{-2}a^2u^2-2s^{-2}a^2u^2+2qc+1=0. \] Thus $cq=-\frac12-s^{-2}a^2u^2$.  Since $c-q= 2s^{-1}au$ we get that $c$ and $-q$ are solutions of the equation $X^2-2s^{-1}au X+\frac12+s^{-2}a^2u^2=0$ and this equation has no real solution. In conclusion the case $s\not=0$ is impossible.
            \item $s=0$.     From the first equation in \eqref{eq} we get $a^2+m^2=0$ and hence $a=m=0$. From the second equation, we get $u=0$, from the third equation we get $p_1=0$,  from the fourth equation we get $z=0$ and from the fifth equation we get $p_2=0$. 
           Then \eqref{eq} is now equivalent to
    	\begin{align*}
     Bk=kx=Ar=rx&= 0,\\
     -cr+kd=-lr+kq &= 0,\\
    r(c+b+q)=k(c+b+q)&= 0,\\
     d(c+q-b)+2rn+(q-c)x&= 0,\\
     l(c+q-b)+2ky+(q-c)x&= 0,\\
     dl-lx+q^2+2ry-qb-dx-B&= 0,\\
    c^2+dl+dx+2kn-cb+lx-A&= 0.
    \end{align*}  
        Then $k=r=0$ and
    \[ \left\{ \begin{array}{ccc}d(c+q-b)+(q-c)x&=& 0,\\
    l(c+q-b)+(q-c)x&=& 0,\\dl-lx+q^2-qb-dx-B&=& 0,\\
        c^2+dl+dx-cb+lx-A&=& 0.\end{array}\right. \]
        This is equivalent to
    \[d(c+q-b)+(q-c)x=0,\;(d-l)(c+q-b)=0,\;c^2+q^2+2dl+1=b(q+c)\esp B=dl-lx+q^2-qb-dx.   \]    
   If $d\not=l$  then $b=c+q$,  $2(dl-qc)=-1$ and $(q-c)x=0$. Since $d$ and $l$ play symmetric roles we can suppose $d\not=0$. We get two types of Lie algebras   
   \[ [e,f]=(d-l)g,\;[e,h]=\e\sqrt{dl+\frac12}e+(l+x)f+ng,\;[f,h]=(d-x)e+\e\sqrt{dl+\frac12}f+yg,\;[g,h]=2\e\sqrt{dl+\frac12}g, d\not=l \] or
   \[ [e,f]=(d-\frac{2qc-1}{2d})g,\;[e,h]=ce+\frac{2qc-1}{2d}f+ng,\;[f,h]=de+qf+yg,\;[g,h]=(c+q)g,d-\frac{2qc-1}{2d}\not=0. \] 
   If $d=l$ then $b(q+c)=c^2+q^2+2l^2+1$ and hence $b+c\not=0$. If $q=c$ then we have two types of Lie algebras
   \[ \;[e,h]=ce+xf+ng,\;[f,h]=-xe+cf+yg,\;[g,h]=\frac{2c^2+1}{2c}g, \]
   \[ \;[e,h]=\e\sqrt{\frac{2l^2+1}{2}}e+(l+x)f+ng,\;[f,h]=(l-x)e+\e\sqrt{\frac{2l^2+1}{2}}f+yg,\;[g,h]=2\e\sqrt{\frac{2l^2+1}{2}}g. \]
    If $q\not=c$ then $b=\frac{2l^2+q^2+c^2+1}{q+c}$ and $x=\frac{2l^3+l-2lqc}{q^2-c^2}$. So     
  \[ [e,h]=ce+(l+\frac{2l^3+l-2lqc}{q^2-c^2})f+ng,\;[f,h]=(d-\frac{2l^3+l-2lqc}{q^2-c^2})e+qf+yg,\;[g,h]=
  \frac{2l^2+q^2+c^2+1}{q+c}g. \]
  For all these models we have $\mathrm{L}_{h,h}\Ku(f,h)\not=0$ which shows that there are not second-order locally symmetric. Moreover, $\mathfrak{h}(\Ku)$ is invariant by $\mathrm{L}$ which shows that $\mathfrak{h}(\Ku)=\mathfrak{h}(\G)$.\qedhere

\end{enumerate}
\end{proof}

\begin{theo} Let $(\G,\langle\;,\; \rangle)$ be a four-dimensional semi-symmetric indecomposable Lorentzian  Lie algebra of type $(S40)$ with $\Ri\neq 0$ and $\dim\mathfrak{h}(\Ku)=1$.  Then, there  exists  a basis $(e,f,g,h)$ with the non vanishing products  $\langle e, e \rangle=\langle f, f \rangle=\langle g, h \rangle=1 $ and the non vanishing brackets are
\[ [e,f]=\frac{2a^2+1}{2a}g,\;[e,h]=\frac{1}{2a(2a^2-1)}f+xg,\;[f,h]=
 \frac{2a(a^2-1)}{2a^2-1}e+yg. \]
 Moreover, $\mathfrak{h}(\G)=\mathrm{span}\{e\wedge g,f\wedge g  \}$ and $\G$ is not second-order locally symmetric.

\end{theo}

   \begin{proof} We proceed as in the proof of Theorem \ref{main2} and we
   suppose  that $A=0$. Then \eqref{bianchi} is equivalent to $u_1=u_2=s=a=0$, $v_2=2w$ and $m=-2k-p_1$. We have
 \[ [\mathrm{R}_e]=\left(\begin{array}{cccc}
 0&0&0&0\\0&-m&0&-x\\-c&-d&-u&-n\\0&-v_1&0&-p_1
 \end{array}  \right)\esp [\mathrm{R}_g]=\left(\begin{array}{cccc}
 0&v_1&0&p_1\\0&2w&0&p_2\\-k&-r&-w&-b\\0&0&0&0
 \end{array}  \right) \]  
  From \eqref{cu2}, we get $[\mathrm{R}_g,\mathrm{L}_g]-\mathrm{R}_g^2-w\mathrm{R}_g=0$ which is equivalent to
  \begin{equation}\label{eq11} w=v_1(z-p_2)=v_1(u+k)=-uk-rz+kp_1+rp_2+up_1+zp_2=0. \end{equation} 
   We have also $[\mathrm{R}_e,\mathrm{L}_e]-\mathrm{R}_e^2-w\mathrm{R}_e=0$ which is equivalent to
   \begin{eqnarray}\label{eq10} &&cv_1=cp_1=-lv_1-m^2+v_1x+cm=-lm-kx+lp_1-xm-p_1x+cx=0,\nonumber\\
  && -ck+c^2=lu-lm-kd-md-du+v_1n+cd=-u^2+uc=-c^2-ld-2kn-lx-xd-nu-p_1n+cn=0,\nonumber\\
   &&(-k+m+p_1-c)v_1=lv_1+v_1x-p_1^2+cp_1=0. \end{eqnarray}
   If $v_1\not=0$ we get from \eqref{eq11} and \eqref{eq10}  $c=u=k=z=p_2=0$ and $m=-p_1$ and \eqref{eq10} becomes
     \begin{eqnarray*} -lv_1-m^2+v_1x=-2lm
       =lu-lm-md+v_1n=-ld-lx-xd-nu+mn=
        lv_1+v_1x-m^2=0. \end{eqnarray*}
       This implies that $l=0$. We return to \eqref{eq} and we find that the first equation implies $m=x=y=n=b=0$ and from the fifth equation we deduce that $q=0$ and $B=0$ which is impossible.

 Thus $v_1=0$ and hence $p_1=0$. From the first equation in \eqref{eq} we get $m=0$, from the second equation we get $u=0$, from the fourth equation we get $z=0$, from the fifth equation we get $p_2=0$. From $m=-2k-p_1$ we get $k=0$ and since $c^2=ck$ we deduce that $c=0$. Thus 
 \[ \mathrm{L}_e=l f\wedge g,\; \mathrm{L}_f=de\wedge g+qf\wedge g-rg\wedge h,\; \mathrm{L}_g=0\esp \mathrm{L}_h=xe\wedge f+ne\wedge g+yf\wedge g-bg\wedge h,   \]
 and \eqref{eq} is now equivalent to
   \[ xr=lr=ld+xd+lx=lq+xq-lb=qd+2rn+xq-db=ld+q^2+2ry-lx-xd-qb-B=r(q+b)=0. \]
 If $x=0$ then one can check easily that $\mathfrak{h}(\Ku)=\mathrm{span}\{f\wedge g \}$ is invariant by $\mathrm{L}$ and hence $\mathfrak{h}(\Ku)=\mathfrak{h}(\G)$ which leaves invariant $\R e$ and hence it is decomposable. Suppose that $x\not=0$. Then
 \[ r=0,ld+xd+lx=lq+xq-lb=qd+xq-db=ld+q^2-lx-xd-qb+1=0. \]
 Since  $xq=lb-lq=db-qd$ we get $(l-d)(q-b)=0$. If $q=b$ then $q=b=0$ 	and hence
 \[ r=0,ld+xd+lx=2ld+1=0. \]
 So $x=\frac{d}{2d^2-1}$ and $l=-\frac1{2d}$. In this case the Lie brackets are
 \[ [e,f]=\frac{2d^2+1}{2d}g,\;[e,h]=\frac{1}{2d(2d^2-1)}f+ng,\;[f,h]=
 \frac{2d(d^2-1)}{2d^2-1}e+yg. \]
   If $q\not=b$ then $l=d$ and
   \[ r=0,\; d=l=\frac{xq}{b-q},l^2+2lx=2l^2+q^2-qb+1=0. \]
   So
   \[ r=0,d=l=\frac{xq}{b-q}, x=-\frac{l}2, b=\frac{2l^2+q^2+1}q. \]
   This implies that $l(1+\frac{q^2}{2l^2+1})=0$ and hence $x=0$. The semi-symmetric Lie algebras obtained satisfy $\mathrm{L}_{h,h}\Ku(f,h)\not=0$ and hence are not second-order locally symmetric.
    \end{proof} 
   
   \begin{pr}\label{pr40} Let $(\G,\langle\;,\; \rangle)$ be a four-dimensional semi-symmetric Lorentzian decomposable Lie algebra of type $(S40)$ with $\Ri\neq 0$. Then $\G$ is a product of $\R$ with a Lie algebra of type 
   $(S30)$.
          
   \end{pr}
   
   \begin{proof} In this case $A=0$ or $B=0$. We suppose $A=0$ and we consider the basis $(e,f,g,h)$ where $\Ku(f,h)=-f\we g$ and $\Ku(e,.)=\Ku(g,.)=0$. Let $E$ be a nondegenerate vector subspace of $\G$ invariant by the holonomy Lie algebra. We can suppose that $\dim E=1$ or $\dim E=2$. If $\dim E=2$ and since $E$ must be invariant by $f\wedge g$ then $E\subset\mathrm{span}\{f,g  \}$ or $E\subset\mathrm{span}\{f,g  \}^\perp$ which is impossible so $\dim E=1$. Let $u$ be a generator of $E$. Since $f\wedge g(u)=0$ then $u\in\mathrm{span}\{e,g  \}$. So $u=e+\al g$. By making the change of basis $(e,f,g,h)$ into $(e+\al g,f,g,h-\al e)$ we can suppose $u=e$. Then the left invariant vector field associated to $e$ must be parallel and hence $\mathrm{R}_e=0$. Hence $\mathrm{span}\{f,g,h \}$ is a semi-symmetric Lie algebra of dimension 3 with isotropic Ricci curvature. According to Proposition \ref{pr30} and its proof
   $\mathrm{L}_f=af\wedge g+cg\wedge h$,  $\mathrm{L}_g=0$ and $\mathrm{L}_h=pf\wedge g+rg\wedge h$ with $(c=0, a^2+ar+1=0)$ or $(c\not=0,p=\frac{2r^2+1}{2c})$. Put $\mathrm{L}_e=xf\wedge g+yf\wedge h+zg\wedge h$. The relation $\Ku(e,f)=0$ is equivalent to
      \[ y(pf\wedge g+rg\wedge h)+xcg\wedge f+yag\wedge h+ycf\wedge h+zaf\wedge g=0. \]If $c=0$ then
   $ y(r+a)=yp+za=0.$
   Since $a\not=-r$ we get $y=z=0$. On the other hand, the relation $\Ku(e,h)=0$ gives
$ xaf\wedge g+xrf\wedge g=0$ and hence $x=0$.
 
 If $c\not=0$ then $y=0$ and $xc=za$. The relation $\Ku(e,h)=0$ gives
\[ x(af\wedge g+cg\wedge h)-xrg\wedge f-zpf\wedge g=0. \] So $x=0$ and $z=0$. Thus $\mathrm{L}_e=0$ which completes the proof.           
            \end{proof}
   
   \section{Four-dimensional Ricci flat or  Ricci isotropic homogeneous semi-symmetric Lorentzian manifolds } \label{section6}
   
   We use Komrakov's classification \cite{komrakov} of four-dimensional homogeneous pseudo-Riemannian manifolds and we apply the following algorithm to find among Komrakov's list the pairs $(\bar{\G},\G)$ corresponding to    four-dimensional Ricci flat or  Ricci isotropic homogeneous semi-symmetric Lorentzian manifolds.

   Let $M=\bar{G}/G$ be an homogeneous manifold with $G$ connected and $\bar{\G}=\G\oplus\mathfrak{m}$, where $\bar{\G}$ is the Lie algebra of $\bar{G}$, $\G$ the Lie algebra of $G$ and $\mathfrak{m}$  an arbitrary complementary of $\G$ (not necessary $\G$-invariant). The pair $(\bar{\G},\G)$ uniquely defines the isotropy representation $\rho:\G\too\mathrm{gl}(\mathfrak{m})$ by $\rho(x)(y)=[x,y]_\mathfrak{m}$, for all $x\in\G$, $y\in\mathfrak{m}$. Let $\{e_1,\ldots,e_r,u_1,\ldots,u_n \}$ be a basis of $\bar{\G}$ where $\{e_i\}$ and $\{u_j \}$ are bases of $\G$ and $\mathfrak{m}$, respectively. The algorithm goes as follows.
   
   \begin{enumerate}\item Determination of invariant pseudo-Riemannian metrics on $M$. It is well-known that invariant pseudo-Riemannian metrics on $M$ are in a one-to-one correspondence with nondegenerate invariant symmetric bilinear forms on $\mathfrak{m}$. A symmetric bilinear form on $\mathfrak{m}$ is determined by its matrix $B$ in $\{u_i\}$ and its invariant if $\rho(e_i)^t\circ B+B\circ\rho(e_i)=0$ for $i=1,\ldots,r$.
   \item Determination of the Levi-Civita connection. Let $B$ be a nondegenerate invariant symmetric bilinear forms on $\mathfrak{m}$. It defines uniquely an invariant linear Levi-Civita connection $\na:\bar{\G}\too\mathrm{gl}(\mathfrak{m})$ given by
   \[ \na(x)=\rho(x),\;\na(y)(z)=\frac12[y,z]_\mathfrak{m}+\nu(y,z),\; x\in\G, y,z\in\mathfrak{m}, \] where $\nu:\mathfrak{m}\times\mathfrak{m}\too\mathfrak{m}$ is given by the formula
   \[ 2B(\nu(a,b),c)=B([c,a]_\mathfrak{m},b)+B([c,b]_\mathfrak{m},a),\;a,b,c\in\mathfrak{m}. \]
   \item Determination of the curvature. The curvature of $B$ is the bilinear map
   $\Ku:\mathfrak{m}\times\mathfrak{m}\too\mathrm{gl}(\mathfrak{m})$ given by
   \[ \Ku(a,b)=[\na(a),\na(b)]-\na([a,b]_\mathfrak{m})-\rho([a,b]_{\G}),\; a,b\in\mathfrak{m}. \]
   \item Determination of the Ricci curvature. It is given by its matrix in $\{ u_i\}$, i.e., $\ric=(\ric_{ij})_{1\leq i,j\leq n}$ where
   \[ \ric_{ij}=\sum_{r=1}^n\Ku_{ri}(u_r,u_j). \]
   \item Determination of the Ricci operator. We have $\Ri=B^{-1}\ric$.
   \item Checking  the semi-symmetry condition. 
   
  \end{enumerate}  
   \begin{theo}\label{main3} Let $M=\bar{G}/G$ be four-dimensional Ricci isotropic homogeneous semi-symmetric Lorentzian manifold. Then $M$ is isometric to one of the following models:
   \begin{enumerate}\item ${\mathbf{1.1}^2}:1$, $\bar{\G}=\mathrm{span}\{e_1,u_1,u_2,u_3,u_4 \}$ with $B_0=\left(
   \begin{array}{cccc}a&0&0&0\\0&0&0&b\\0&0&a&0\\0&b&0&d\end{array}  \right)$ $(ab\not=0)$ and
   \[ [e_1,u_1]=u_3,[e_1,u_3]=-u_1,[u_1,u_3]=-u_2,[u_1,u_4]=u_1,[u_2,u_4]=2u_2,[u_3,u_4]=u_3, \]
   \item ${\mathbf{1.1}^2}:2$, $\bar{\G}=\mathrm{span}\{e_1,u_1,u_2,u_3,u_4 \}$ with $B_0=$  and
      \[ [e_1,u_1]=u_3,[e_1,u_3]=-u_1,[u_1,u_4]=u_1,[u_2,u_4]=pu_2,[u_3,u_4]=u_3,
      \quad p\not=1. \]
   \item ${\mathbf{1.1}^2}:5$, $\bar{\G}=\mathrm{span}\{e_1,u_1,u_2,u_3,u_4 \}$ with $B_0$ and
         \[ [e_1,u_1]=u_3,[e_1,u_3]=-u_1,[u_1,u_3]=u_2. \]
   \item ${\mathbf{1.4}^1}:2$, $\bar{\G}=\mathrm{span}\{e_1,u_1,u_2,u_3,u_4 \}$ with $B_1=\left(
         \begin{array}{cccc}0&0&-a&0\\0&a&0&0\\-a&0&b&d\\0&0&d&c\end{array}  \right)$ $(ac\not=0)$ and
         \[ [e_1,u_2]=u_1,[e_1,u_3]=u_2,[e_1,u_4]=e_1,[u_1,u_4]=u_1,[u_3,u_4]=-u_3. \]
   \item ${\mathbf{1.4}^1}:9$, $\bar{\G}=\mathrm{span}\{e_1,u_1,u_2,u_3,u_4 \}$ with $B_1$ and
            \[ [e_1,u_2]=u_1,[e_1,u_3]=u_2,[u_1,u_3]=u_1,[u_2,u_3]=re_1+u_2+u_4,[u_3,u_4]=pu_4\quad c+2ra+2ap^2+2pa\not=0. \]
    \item ${\mathbf{1.4}^1}:10$, $\bar{\G}=\mathrm{span}\{e_1,u_1,u_2,u_3,u_4 \}$ with $B_1$ and
                \[ [e_1,u_2]=u_1,[e_1,u_3]=u_2,[u_1,u_3]=u_1,[u_2,u_3]=re_1+u_2,[u_3,u_4]=pu_4\quad r+p^2+p\not=0. \]
      \item ${\mathbf{1.4}^1}:11$, $\bar{\G}=\mathrm{span}\{e_1,u_1,u_2,u_3,u_4 \}$ with $B_1$ and
                  \[ [e_1,u_2]=u_1,[e_1,u_3]=u_2,[u_1,u_3]=u_1,[u_2,u_3]=re_1+u_2+u_4,[u_3,u_4]=u_1-u_4\quad c+2ra\not=0. \]
                            
        \item ${\mathbf{1.4}^1}:12$, $\bar{\G}=\mathrm{span}\{e_1,u_1,u_2,u_3,u_4 \}$ with $B_1$ and
                          \[ [e_1,u_2]=u_1,[e_1,u_3]=u_2,[u_1,u_3]=u_1,[u_2,u_3]=re_1+u_2,[u_3,u_4]=u_1-u_4\quad r\not=0. \]
         \item ${\mathbf{1.4}^1}:13$, $\bar{\G}=\mathrm{span}\{e_1,u_1,u_2,u_3,u_4 \}$ with $B_1$ and
                           \[ [e_1,u_2]=u_1,[e_1,u_3]=u_2,[u_2,u_3]=re_1+u_4,[u_3,u_4]=u_4\quad c+2ra+2a\not=0. \]
                                            
       \item ${\mathbf{1.4}^1}:14$, $\bar{\G}=\mathrm{span}\{e_1,u_1,u_2,u_3,u_4 \}$ with $B_1$ and
                                  \[ [e_1,u_2]=u_1,[e_1,u_3]=u_2,[u_2,u_3]=re_1,[u_3,u_4]=u_4\quad r+1\not=0. \]
                                                                              
 \item ${\mathbf{1.4}^1}:15,16\;\mbox{and}\;  17$, $\bar{\G}=\mathrm{span}\{e_1,u_1,u_2,u_3,u_4 \}$ with $B_1$ and
                            \[ [e_1,u_2]=u_1,[e_1,u_3]=u_2,[u_2,u_3]=\e e_1+u_4,[u_3,u_4]=u_1\quad  c+2\e a\not=0,\;\e=0,1,-1. \]
    \item ${\mathbf{1.4}^1}:18,19\;\mbox{and}\;  20$, $\bar{\G}=\mathrm{span}\{e_1,u_1,u_2,u_3,u_4 \}$ with $B_1$ and
                               \[ [e_1,u_2]=u_1,[e_1,u_3]=u_2,[u_2,u_3]=\e e_1+u_4,\quad  c+2\e a\not=0,\;\e=0,1,-1. \]
     \item ${\mathbf{1.4}^1}:21\;\mbox{and}\;  22$, $\bar{\G}=\mathrm{span}\{e_1,u_1,u_2,u_3,u_4 \}$ with $B_1$ and
                                    \[ [e_1,u_2]=u_1,[e_1,u_3]=u_2,[u_2,u_3]= \e e_1,[u_3,u_4]=u_1 ,\;\e=1,-1. \]
      \item ${\mathbf{1.4}^1}:24\;\mbox{and}\;  25$, $\bar{\G}=\mathrm{span}\{e_1,u_1,u_2,u_3,u_4 \}$ with $B_1$ and
                                          \[ [e_1,u_2]=u_1,[e_1,u_3]=u_2,[u_2,u_3]= \e e_1 ,\;\e=1,-1. \]
     \item ${\mathbf{2.5}^2}:2$, $\bar{\G}=\mathrm{span}\{e_1,e_2,u_1,u_2,u_3,u_4 \}$ with $B_2=\left(
                                               \begin{array}{cccc}0&0&a&0\\0&a&0&0\\a&0&b&0\\0&0&0&a\end{array}  \right)$ $(a\not=0)$ and
  \[[e_1,u_2]=u_1,[e_1,u_3]=-u_2,[e_2,u_3]=u_4,[e_2,u_4]=-u_1,[u_1,u_3]= u_1,[u_2,u_3]=A,[u_2,u_4]=2ru_1,[u_3,u_4]=B, \]
  where $A=(p+s)e_1+re_2+u_2-2ru_4$, $B=-re_1+(p-s)e_2-2ru_2-u_4$, $r\geq0,s\geq0$ and $p+r^2\not=0$.
 \item ${\mathbf{2.5}^2}:3$, $\bar{\G}=\mathrm{span}\{e_1,e_2,u_1,u_2,u_3,u_4 \}$ with $B_2$ and
  \[[e_1,u_2]=u_1,[e_1,u_3]=-u_2,[e_2,u_3]=u_4,[e_2,u_4]=-u_1,
  [u_2,u_3]=-(r+s)e_1-u_4,
  [u_2,u_4]=u_1,[u_3,u_4]=(s-r)e_2-u_2, \]
  with $4r-1\not=0$ and $s\geq0$. 
  
\item ${\mathbf{2.5}^2}:4\;\mbox{and}\;5$, $\bar{\G}=\mathrm{span}\{e_1,e_2,u_1,u_2,u_3,u_4 \}$ with $B_2$ and
  \[[e_1,u_2]=u_1,[e_1,u_3]=-u_2,[e_2,u_3]=u_4,[e_2,u_4]=-u_1,
  [u_2,u_3]=\e(1+s)e_1,[u_3,u_4]=\e(1-s)e_2,\;s\geq0,\;\e=-1,1. \]
                                                                                                                                                                                                                             
   \end{enumerate}

   \end{theo}

   In Table 1, we give the list of Ricci flat homogeneous semi-symmetric Lorentzian manifolds.
   
   \begin{center}
   \begin{tabular}{|l||l||l||l|}
     \hline
     Index &$ B_0$& Conditions&  $ dim(\h(\Ku))$\\
     \hline
                                  $\begin{array}{c}
                                    1.1^2:2,8\\
                                    1.1^2:12 \\
                                    1.1^3:1 \\
                                    1.1^4:1 \\
                                  \end{array}
                                $& $\left(
                                  \begin{array}{cccc}
                                    a& 0 &0 & 0 \\
                                    0 & b & 0 & c \\
                                    0 & 0 & a & 0 \\
                                    0 & c & 0 & d \\
                                  \end{array}
                                \right)$, $ bd< c^2$ & $\begin{array}{c}
                                    p=0\\
                                    \la=0 \\
                                              \\
                                              \\                      \end{array}
                                $&$\begin{array}{c}
                                    0\\
                                    0 \\
                                  0\\
                                    0
                                    \\
                                    \end{array}
                                $\\
     \hline
      $
                                  \begin{array}{c}
                                    1.4^1:2\\
                                    1.4^1:9 \\
                                     1.4^1:10\\
                                    1.4^1:11 \\
                                    1.4^1:12 \\
                                    1.4^1:13 \\
                                     1.4^1:14\\
                                    1.4^1:16,19 \\
                                    1.4^1:23,26 \\
                                  \end{array}
                                $& $\left(
                                  \begin{array}{cccc}
                                    0& 0 & a & 0 \\
                                    0 & -a & 0 & 0 \\
                                    a & 0 & b & d \\
                                    0 & 0 & d & c \\
                                  \end{array}
                                \right)$,  $ac<0$ & $\begin{array}{c}
                                    (p,b)=(1,0)\\
                                     
                                     2a(p^2-p+r)=c\\
                                   p^2+p+r=0 \\
                                    2ar=c \\
                                   r=0 \\
                                    2a(r+1)=c\\
                                    r=-1 \\
                                    2a+c=0 \\
                                     \\
                                  \end{array}$&$ \begin{array}{c}
                                   0\\
                                      2\\
                                     0~~ ou~~ 2 \\
                                   2 \\
                                    0 \\
                                    2\\
                                   2 \\
                                     2\\0\\
                                  \end{array}$\\
     \hline
                  $\begin{array}{c}
                                   2.1^2:6 \\
                                  \end{array}
                                $& $\left(
                                  \begin{array}{cccc}
                                    0& 0 &a & 0 \\
                                    0 & b & 0 & 0 \\
                                    a & 0 & 0 & 0 \\
                                    0 & 0 & 0 & b \\
                                  \end{array}
                                \right)$ & $\begin{array}{c}
                                   
                                    \\
                                    \end{array}
                                $&$\begin{array}{c}
                                    
                                                                   0 \\
                                  \end{array}$ \\
     \hline
         $\begin{array}{c}
                                    2.4^1:3 \\
                                  \end{array}
                                $& $\left(
                                  \begin{array}{cccc}
                                    0& 0 &a & 0 \\
                                    0 & -a & 0 & 0 \\
                                    a & 0 & 0 & 0 \\
                                    0 & 0 & 0 & b \\
                                  \end{array}
                                \right),$ $ab<0$&
                               &0 \\
     \hline
         $\begin{array}{c}
                                    2.5^2:2 \\
                                   2.5^2:3 \\
                                    2.5^2:6 \\
                                   2.5^2:7 \\
                                  \end{array}
                                $& $\left(
                                  \begin{array}{cccc}
                                    0& 0 &a & 0 \\
                                    0 & a & 0 & 0 \\
                                    a & 0 & b & 0 \\
                                    0 & 0 & 0 & a \\
                                  \end{array}
                                \right)$ & $\begin{array}{c}
                                    p=-r^2 \\
                                    4r=1\\
                                    \\
                                    \\
                                    \end{array}
                                $&$\begin{array}{c}
                                   0~~ si~~ s=0~~ ou~~ 2~~ si~~ s\neq0 \\
                                   0~~ si~~ s=0~~ ou~~2~~ si~~ s\neq0 \\2\\
                                                                   0 \\
                                  \end{array}$ \\
     \hline

                  $\begin{array}{c}
                                   3.2^2:2 \\
                                  \end{array}
                                $& $\left(
                                  \begin{array}{cccc}
                                    0& 0 &a & 0 \\
                                    0 & a & 0 & 0 \\
                                    a & 0 & 0 & 0 \\
                                    0 & 0 & 0 & a \\
                                  \end{array}
                                \right)$ &
                                &$  0$ \\
                                  
     \hline
         $
                                    3.5^1:4 
                                $& $\left(
                                  \begin{array}{cccc}
                                    0& 0 &2a & 0 \\
                                    0 & a & 0 & 0 \\
                                    2a & 0 & 0 & 0 \\
                                    0 & 0 & 0 & b \\
                                  \end{array}
                                \right),$ $ab>0$&
                                &$
                                                                   0 $ \\
     \hline
         $
                                    3.5^2:4 $& $\left(
                                  \begin{array}{cccc}
                                   a& 0 &0 & 0 \\
                                    0 & a & 0 & 0 \\
                                   0 & 0 & a & 0 \\
                                    0 & 0 & 0 & b \\
                                  \end{array}
                                \right),$ $ab<0$&
                                &$
                                    0 $ \\
     \hline
         $
                                    4.1^2:1
                                $& $\left(
                                  \begin{array}{cccc}
                                    0& 0 &a & 0 \\
                                    0 & a & 0 & 0 \\
                                    a & 0 & 0 & 0 \\
                                    0 & 0 & 0 & a \\
                                  \end{array}
                                \right)$&
                                &$
                                                                   0 $ \\
     \hline
         $
                                    6.1^3:3  $& $\left(
                                  \begin{array}{cccc}
                                   a& 0 &0 & 0 \\
                                    0 & a & 0 & 0 \\
                                   0 & 0 & a & 0 \\
                                    0 & 0 & 0 & -a \\
                                  \end{array}
                                \right)$&
                                  &$0$ \\
     \hline
     \end{tabular}\bigskip
     
     (Table 1: List of Ricci flat  homogeneous semi-symmetric Lorentzian manifolds)
   \end{center}

 \section{Proof of Theorem \ref{main}}\label{section7}
 
 The proof is based on the results of Sections \ref{section2} and \ref{section3} and the following two theorems proved, respectively, in \cite{derd} and \cite{calvaruso3}.
 
 \begin{theo}[\cite{derd}]\label{derd} Let $(M,g)$ be an oriented four-dimensional Lorentzian Einstein manifold whose curvature operator, treated as a complex-linear vector bundle morphism	$\wi K	:	\wedge^2TM\too \wedge^2TM$,	is	diagonalizable	at	every	point	and	has	complex eigenvalues that form constant functions $M\too \C$. Then $(M,g)$ is locally homogeneous, and one of the following three cases occurs:\begin{enumerate}\item[$(a)$]
 $(M, g)$ is a space of constant curvature. \item[$(b)$] $(M, g)$ is locally isometric to the Riemannian product of two 
 pseudo-Riemannian surfaces having the same constant Gaussian curvature. \item[$(c)$] $(M, g)$ is locally isometric to  a Petrov's Ricci-flat manifold.\end{enumerate}   Furthermore, $(M, g)$ is locally symmetric in cases $(a)-(b)$, but not in $(c)$, and in
case $(c)$ it is locally isometric to a Lie group with a left-invariant metric.
 
 \end{theo}
 
 \begin{theo}[\cite{calvaruso3}]\label{calv} Let $(M,g)$ be a locally homogeneous Lorentzian four-manifold. If its Ricci operator is diagonalizable then $(M,g)$ is either Ricci-parallel or locally isometric to a Lie group equipped with a left invariant Lorentzian metric.

 \end{theo}
 
 \paragraph{Proof of Theorem \ref{main}} \begin{proof} Let $(M,g)$ be a locally homogeneous semi-symmetric Lorentzian four-manifold. According to Proposition \ref{pr2}, its Ricci operator has only real eigenvalues. Suppose that $\Ri$   has a non null eigenvalue.  Then this eigenvalue has at least multiplicity two and $\Ri$ must be diagonalizable. So, according to Theorem \ref{calv}, $(M,g)$ is either Ricci-parallel or locally isometric to a Lie group equipped with a left invariant Lorentzian metric. If $(M,g)$ is Ricci-parallel and has two distinct eigenvalues then, according to Theorem 7.3 in \cite{calvaruso3}, one of the situation $(a)-(d)$ occurs. Suppose now that $(M,g)$ is Einstein with non null scalar curvature. According to Theorem \ref{main2}, the total curvature is diagonalizable and we can apply Theorem \ref{derd} to get that one of the conclusions $(a)$ or $(b)$ of this theorem occurs. Now, if $\Ri$ has only 0 as eigenvalue then, according to Proposition \ref{pr2}, $\Ri^2=0$. This completes the proof.\end{proof}

\end{document}